\documentclass[10pt]{article}
\usepackage{amsmath,amssymb,amsthm,enumerate,color,bbm,a4wide}
\usepackage[colorlinks]{hyperref}

\def\Nset{\mathbb{N}}
\def\Zset{\mathbb{Z}}
\def\Rset{\mathbb{R}}
\def\mbc{\mathcal{C}}
\def\rme{\mathrm{e}}

\def\rmd{\mathrm{d}} 
\def\esp{\mathbb E}
\def\pr{\mathbb P}

\def\cov{\mathrm{cov}}

\def\convvague{\stackrel{\mbox{\tiny\rm v}}{\to}} %Vague convergence
\def\eqdistr{\stackrel{\mbox{\tiny\rm d}}{=}}

\newcommand{\CTE}{{\rm CTE}}
\newcommand{\ES}{{\rm ES}}
\newcommand{\VAR}{{\rm VAR}}
\newcommand\1[1]{\mathbbm{1}_{#1}}
\newcommand{\sequence}[3]{\{#1_#2,#2\in#3\}}
\newcommand{\scalfunccev}{b}
\newcommand{\cevcdf}{\boldsymbol{\Psi}}
\newcommand{\hpi}[2]{\boldsymbol{\Pi}_{#1,#2}}  %%% transition kernel used in the proof of theorem for Markov chains
\newcommand\nucond{\boldsymbol{\mu}}
\newcommand{\numult}[1]{\boldsymbol{\nu}_{#1}}

\newtheorem{theorem}{Theorem}[section]
\newtheorem{lemma}[theorem]{Lemma}

\newtheorem{proposition}[theorem]{Proposition}
\newtheorem{definition}{Definition}
\newtheorem{hypothesis}{Assumption}

\theoremstyle{remark}
\newtheorem{remark}[theorem]{Remark}

\numberwithin{equation}{section}

\begin{document}
\title{Heavy tailed time series with extremal independence}

\author{Rafa{\l} Kulik\thanks{University of Ottawa} \and Philippe
  Soulier\thanks{Universit\'e de Paris-Ouest}}

\date{}
\maketitle

\begin{abstract}
  We consider heavy tailed time series whose finite-dimensional distributions
  are extremally independent in the sense that extremely large values cannot be
  observed consecutively. This calls for methods beyond the classical
  multivariate extreme value theory which is convenient only for extremally
  dependent multivariate distributions. We use the Conditional Extreme Value
  approach to study the effect of an extreme value at time zero on the future of
  the time series. In formal terms, we study the limiting conditional
  distribution of future observations given an extreme value at time zero.  To
  this purpose, we introduce conditional scaling functions and conditional
  scaling exponents. We compute these quantities for a variety of models,
  including Markov chains, exponential autoregressive models, stochastic
  volatility models with heavy tailed innovations or volatilities.
\end{abstract}


\section{Introduction}
\label{sec:intro}
Let $\sequence{X}{t}{\Zset}$ be a strictly stationary time series. We say that
$\{X_t\}$ is regularly varying if all its finite dimensional distributions are
regularly varying, i.e.~for each $h\geq0$, there exists a nonzero Radon measure
$\numult{h}$ on $[-\infty,\infty]^{h+1}\setminus\{\boldsymbol{0}\}$, called the
exponent measure, which puts zero mass at infinity, and a scaling function
$c$ such that, as $s\to\infty$,
\begin{align}
  \label{eq:def-mrv}
  s \pr \left( \frac{(X_0,\dots,X_h)}{c(s)} \in \cdot \right)
  \convvague \numult{h} \; ,
\end{align}
where $\convvague$ means vague convergence, to be understood here on the space
$[-\infty,\infty]^{h+1}\setminus\{\boldsymbol{0}\}$. Recall that a sequence of measures $\nu_n$
defined on a complete separable metric space~$E$ (endowed with its Borel $\sigma$-field) is said to
converge vaguely to a measure $\nu$ if $\nu_n(f)\to\nu(f)$ for all continuous functions with compact
support, or equivalently $\nu_n(K)\to\nu(K)$ for all compact sets $K$ with $\nu(\partial K)=0$. See
\cite{resnick:1987} for more details.  This assumption implies that the function $c$ is regularly
varying with index $1/\alpha$ for some $\alpha>0$, the measure~$\numult{h}$ is homogeneous of degree
$-\alpha$ and the marginal distribution of $X_0$ is heavy tailed with positive tail index~$\alpha$.
% and satisfies the so-called tail
% balance condition: there exists $p\in[0,1]$ such that
% \begin{align*}
%   p = \lim_{x\to\infty} \frac{\pr(X_0>x)}{\pr(|X_0|>x)} = 1 - \lim_{x\to\infty}
%   \frac{\pr(X_0<-x)}{\pr(|X_0|>x)} \; .
% \end{align*}
% The parameter $p$ is the skewness parameter of the marginal
% distribution.
To avoid trivialities, we will only consider distributions that are not totally
skewed to the left, that is we assume that $\lim_{x\to\infty}
\pr(X_0>x)/\pr(|X_0|>x)>0$. In that case, a possible choice for the scaling
function in~(\ref{eq:def-mrv}) is $c(s)=F_0^\leftarrow(1-1/s)$, where $F_0$ is
the distribution function of $X_0$, and we can rewrite~(\ref{eq:def-mrv}) as
\begin{align}
  \label{eq:mrv-x}
  \frac{\pr\left( x^{-1}(X_0,\dots,X_h) \in \cdot \right)}
  {\pr(X_0>x)} \convvague \numult{h} \; ,
\end{align}
on $[-\infty,\infty]^{h+1}\setminus\{\boldsymbol{0}\}$, as $x\to\infty$.

If $h\geq 1$, there exist two fundamentally different cases: either the exponent
measure is concentrated on the axes or it is not. The former case is referred to
as extremal independence and the latter as extremal dependence. In other words,
extremal independence means that no two components can be extremely large at the
same time, and extremal dependence means that two components can be
simultaneously extremely large.  The suitably renormalized componentwise maxima
of an i.i.d.~sequence of extremally independent random vectors converge to a
max-stable distribution with independent marginals. See e.g.~\cite{resnick:2002}
or \cite{resnick:2007}. This definition is rather weak and it must be noted that
if two components of the vector are extremally dependent then the whole vector
is extremally dependent; e.g. the vector $(X,X,Y)$, if multivariate regularly
varying (MRV), is extremally dependent even if $X$ and $Y$ are independent.
For most time series models, the distribution of $(X_0,\dots,X_h)$ is either extremally dependent or
extremally independent for all $h$.

In a time series context, we may want to assess the influence of an extreme
event at time zero on future observations. If the finite dimensional
distributions of the time series model under consideration are extremally
independent or more generally if the vector $(X_0,X_m,\dots,X_h)$ is extremally
independent for some $m\geq1$, then, for any set $A$ which is bounded away from
zero in $\Rset^{h-m+2}$,
\begin{align}
  \label{eq:extremal-indep-joint-exceed}
  \lim_{x\to\infty} \frac{ \pr(X_0>xu_0, (X_m\dots,X_h) \in xA)}{\pr(X_0>x)} = 0
  \; .
\end{align}
Thus in case of extremal independence the exponent measure $\numult{h}$ provides
no information on (most) extreme events occurring after an extreme event at
time~0.  In concrete terms, if the series $\{X_t\}$ represent financial losses,
extremal independence means that an extreme loss at time zero will be followed
by another extreme loss of at least the same magnitude with an extremely small
probability. This is good news, but it is still of great importance to know how
likely a one million euro loss is to be followed by a smaller loss of, say, a
hundred thousand euros, which might be disastrous after the previous loss. A
moderate flood can still be devastating after a major one.  Since the exponent
measure provides no information, other tools must be used to quantify the
influence of an extreme event at time zero on future events.

In order to obtain a non degenerate limit
in~(\ref{eq:extremal-indep-joint-exceed}) and a finer analysis of the sequence
of extreme values, it is necessary to change the normalization
in~(\ref{eq:mrv-x}), and possibly the space on which we will assume that vague
convergence holds. One idea is to find a sequence of normalizations $b_j(x)$,
$j\geq1$ such that for each $h\geq1$, the conditional distribution of
$(X_0/x,X_1/b_1(x),\dots,X_h(x)/b_h(x))$ given $X_0>x$ has a non degenerate
limit. Finding a limiting conditional distribution of a random vector given one
extreme component is a very old problem.  It was recently rigorously
investigated for bivariate distributions using the concept of regular variation
on cones by \cite{heffernan:resnick:2007} and \cite{das:resnick:2011}. See the
references in these paper for the earlier literature.  If such a limiting
distribution exists, the vector $(X_0,\dots,X_h)$ is said to satisfy the
``Conditional Extreme Value'' (hereafter CEV) assumption.  This expression was
introduced in this context by \cite{das:resnick:2011}\footnote{Note that this
  expression is also used in the extreme value literature to refer to the
  standard extreme value theory in presence of a covariate.}.  It must be noted
that if they exist, limiting conditional distributions given an extreme
component need not be extreme value distributions and the variables
$X_1,\dots,X_h$ need not be in the domain of attraction of an extreme value
distribution.  Considering i.i.d.~random variables, it is easily seen that any
distribution can arise as a limiting distribution. Another important feature of
the CEV approach is that it is applicable to extremally dependent regularly
varying multivariate distributions as well; in that case, the limiting
distribution is entirely determined by the exponent measure. This is important
in view of statistical inference, since the same methodology can be applied to
both types of distributions.

It is the goal of this paper to apply the CEV approach to the finite dimensional
distributions of stationary (and some non stationary) time series, both
extremally dependent and independent, though with a main focus on extremally
independent time series.

In Section~\ref{sec:consequences}, we will state Assumption~\ref{hypo:conditional-indep-measure}
which expresses the CEV condition in the correct vague convergence framework introduced by
\cite{heffernan:resnick:2007} and give several applications. One important problem with extremally
independent random variables is the tail behavior of their product. The CEV condition alone does not
guarantee that the product is regularly varying. In Section~\ref{sec:tail-product}, we will
strengthen it and obtain a result on the tail behavior of products.  In
Section~\ref{sec:conv-moment}, we will further strengthen it to obtain the convergence of
conditional moments. This convergence can be applied to study risk measures such as the conditional
tail expectation; this will be discussed in Section~\ref{sec:CTE}.  In
Section~\ref{sec:tail-process} we extend the tail process, introduced by \cite{basrak:segers:2009}
in the extremally dependent case, to the extremally independent case and we give a representation of
the limiting conditional distributions in terms of this tail process.  In Section~\ref{sec:hrv} we
will compare the CEV approach to Hidden Regular Variation.

In the following Sections~\ref{sec:markov-chains}, \ref{sec:explin}, and
\ref{sec:sv} we will study several models of regularly varying time series which
satisfy Assumption~\ref{hypo:conditional-indep-measure}.  In
Section~\ref{sec:markov-chains}, we give a general result for extremally
independent Markov chains.
In Section~\ref{sec:explin}, we study a non-Markovian exponential linear process
and in Section~\ref{sec:sv}, we will finally consider stochastic volatility
models with light tailed or heavy tailed volatilities.  It should be pointed out
that the models studied in Sections~\ref{sec:explin} and~\ref{sec:sv} allow for
some form of long memory. This is of practical importance since it is one of the
so-called stylized facts of financial time series (log-returns) that volatility
may exhibit long memory.

Section~\ref{sec:prooftheomarkov} contains the proof of our main result on
Markov chains and Section~\ref{sec:concluding} discusses some directions of
further research, the most important one being statistical inference.

\section{Limiting conditional distributions and conditional scaling exponents}
\label{sec:consequences}

We now introduce the main Assumption of this paper.  It is stated in terms of
regular variation on space smaller than $[-\infty,\infty]^{h+1}\setminus
\{\boldsymbol{0}\}$.

\begin{hypothesis}
  \label{hypo:conditional-indep-measure}

  There exist scaling functions $b_j$, $j\geq1$ and Radon measures $\nucond_h$,
  $h\geq1$, on $(0,\infty]\times[-\infty,+\infty]^h$, $h\geq1$, such that
  \begin{align}
    \label{eq:indep-vague-noncentered}
    \frac1{\pr(X_0>x)} \pr \left( \left(\frac{X_0}{x},
        \frac{X_1}{\scalfunccev_{1}(x)}, \cdots, \frac{X_h}{\scalfunccev_{h}(x)}
      \right) \in \cdot \right) \convvague \nucond_h \; ,
  \end{align}
  on $(0,\infty]\times[-\infty,+\infty]^h$ and for all $y_0>0$,
 \begin{enumerate}[a.]
 \item the measure $\nucond_h([y_0,\infty]\times\cdot)$ on $\Rset^h$ is not
   concentrated on a line through infinity;  \label{it:not-too-small}
\item the measure $\nucond_h([y_0,\infty] \times \cdot)$ on
   $\Rset^h$ is not concentrated on a hyperplane; \label{it:not-too-large}
\item the measure $\nucond_h(\cdot \times \Rset^h)$ on $(0,\infty]$ is not
   concentrated at infinity. \label{it:reg-not-slow}
\end{enumerate}

\end{hypothesis}

Notice that vague convergence here must hold on a different space than
in~(\ref{eq:def-mrv}). This is of importance since the compact sets of
$[-\infty,\infty]^{h+1}\setminus\{\boldsymbol{0}\}$ and
$(0,\infty]\times[-\infty,+\infty]^h$ differ. For instance, if $h=1$,
$[0,\infty]\times[1,\infty]$ is compact in
$[-\infty,\infty]^2\setminus\{\boldsymbol{0}\}$ but not in
$(0,\infty]\times[-\infty,+\infty]$. More generally, a subset $K$ of
$[-\infty,\infty]^{h+1}\setminus\{\boldsymbol{0}\}$ is relatively compact if
there exists $\epsilon>0$ such that $\boldsymbol{x}\in K$ implies that at least
one component of $\boldsymbol{x}$ is greater than $\epsilon$; a subset $L$ of
$(0,\infty]\times[-\infty,\infty]^h$ is relatively compact if there exists
$\epsilon>0$ such that $\boldsymbol{x}\in L$ implies that the first component of
$\boldsymbol{x}$ is greater than $\epsilon$.

For $h=1$, Assumption~\ref{hypo:conditional-indep-measure} is Condition (5) in
\cite{heffernan:resnick:2007}.  We extend it here to a multidimensional
framework and to different scaling functions $b_j$, $j\geq1$.  This is a
fundamental necessity in the time series context. As already mentioned in the
introduction, Assumption~\ref{hypo:conditional-indep-measure} does not require
stationarity of the time series $\{X_t\}$ and is compatible with
both extremal dependence and independence.

We now make some comments on the conditions in
Assumption~\ref{hypo:conditional-indep-measure}.
\begin{itemize}
\item
  Assumption~\ref{hypo:conditional-indep-measure}\ref{it:not-too-small}
  implies that the scaling functions $b_j$ are not too small.
\item
  Assumption~\ref{hypo:conditional-indep-measure}\ref{it:not-too-large}
  implies that the scaling functions $b_j$ are not too large. For
  instance, if $X_0$ and $X_1$ are independent, choosing $b_1(x)=x$
  would yield a measure concentrated on $(0,\infty]\times\{0\}$.
\item Assumption~\ref{hypo:conditional-indep-measure}\ref{it:reg-not-slow} implies that the marginal
  distribution of $X_0$ is heavy tailed  with positive tail index. Hereafter,
  we let $\alpha$ denote the tail index.
\item By construction, $\nucond_h([1,\infty)\times\Rset^h)=1$,
  i.e. the measure $\nucond_h$ restricted to $[1,\infty)\times\Rset^h$
  is a probability measure. Thus we can define the multivariate
  distribution functions $\cevcdf_h$ on $[1,\infty)\times\Rset^h$ by
\begin{align}
  \label{eq:def-cevcdf}
  \cevcdf_h(\boldsymbol{y}) = \nucond_h \left( [1,y_0] \times \prod_{j=1}^h
    [-\infty,y_j]\right) \; ,
  \end{align}
where $\boldsymbol{y}=(y_0,y_1,\ldots,y_h)\in[1,\infty)\times\Rset^h$.  For all
continuity points $\boldsymbol{y}$ of $\cevcdf_h$, we obtain
\begin{align}
   \label{eq:lim-cevcdf}
   \cevcdf_h(\boldsymbol{y})= \lim_{x\to\infty} \pr \left( \frac{X_0}x \leq
     y_0,\frac{X_1}{b_1(x)} \leq y_1, \dots, \frac{X_h}{\scalfunccev_h(x)} \leq
     y_h \mid X_0>x \right) \; .
\end{align}

\end{itemize}

The most important consequence of
Assumption~\ref{hypo:conditional-indep-measure}, is that the functions $b_j$,
$j\geq 1$ are regularly varying and that the limiting measure $\nucond_h$ has
some homogeneity property. We state these properties as a lemma whose proof is a
standard application of the Convergence to Type Theorem. See
\cite[Proposition~1]{heffernan:resnick:2007}.
\begin{lemma}
  \label{lem:homogeneity}
  If Assumption~\ref{hypo:conditional-indep-measure} holds, then there exists
  $\kappa_j\in\Rset$ such that
  \begin{align*}
    \lim_{t\to\infty} \frac{b_j(ty)} {b_j(t)} = y^{\kappa_j} \;
  \end{align*}
  and for all $y_0>0$ and $(y_1,\dots,y_h)\in\Rset^h$,
  \begin{align}
    \label{eq:homogeneity-cev}
    \nucond_h \left( (ty_0,\infty] \times \prod_{i=1}^h
      [-\infty,t^{\kappa_i}y_i] \right) = t^{-\alpha} \nucond_h \left(
      (y_0,\infty] \times \prod_{i=1}^h [-\infty,y_i] \right) \; .
  \end{align}

\end{lemma}

To put emphasis on the regular variation of the functions $b_j$, we introduce the following
definition.
\begin{definition}[Conditional scaling exponent]
  \label{defi:CSE}
  Under Assumption~\ref{hypo:conditional-indep-measure}, for $h\geq1$, we call the index $\kappa_h$
  of regular variation of the functions $\scalfunccev_h$ the (lag $h$) {\bf conditional scaling
    exponent}.
\end{definition}
The exponents $\kappa_h$, $h\geq1$ reflect the influence of an extreme event at time zero on future
lags. Even though we expect this influence to decrease with the lag in the case of extremal
independence, these exponents are not necessarily monotone decreasing.  See
Sections~\ref{sec:explin} and~\ref{sec:sv-heavy-innovation-leverage}.

Considering only the bivariate distribution of $(X_0,X_h)$, we have the following properties.
\begin{itemize}
\item If $(X_0,X_h)$ is multivariate regularly varying in the sense of (\ref{eq:def-mrv}) and
  Assumption~\ref{hypo:conditional-indep-measure} holds, then $\kappa_h\leq1$. If $(X_0,X_h)$ is
  extremally dependent then $\kappa_h=1$. If $b_h(x)=o(x)$, which holds in particular if
  $\kappa_h<1$, then $(X_0,X_h)$ is extremally independent. Negative values of $\kappa_h$ are
  allowed. This means that extremely large values are typically followed by extremely small
  (absolute) values.
\item Condition (\ref{eq:def-mrv}) and extremal independence do not imply that
  Assumption~\ref{hypo:conditional-indep-measure} holds, i.e. the existence of limiting conditional
  distributions. See Sections~\ref{sec:hrv} and~\ref{sec:counter}.
\item In most of the examples investigated in the next sections, it will hold that $0\leq\kappa_h<1$
  for all $h$. However, there are natural examples where the scaling exponent is larger than~1. See
  Section~\ref{sec:expar}.

\end{itemize}

\subsection{Tail of products}
\label{sec:tail-product}
One application of Assumption~\ref{hypo:conditional-indep-measure} is to obtain
the tail of products of regularly varying random variables.  If a pair
$(X_0,X_h)$ is jointly regularly varying with tail index $\alpha$ and is
extremally dependent, then it is well known that the tail of the product
$X_0X_h$ is $\alpha/2$; see e.g.~\cite[Proposition~7.6]{resnick:2007}.  In the
case of extremal independence, many different tail behaviors of the product are
possible.  Under Assumption~\ref{hypo:conditional-indep-measure} and an
additional technical condition, we can obtain the tail index of $X_0X_h$. The
next result generalizes \cite[Theorem~3.1]{maulik:resnick:rootzen:2002} who
consider only the case $\kappa_h=0$; see also \cite{hazra:maulik:2011}.
As before, we denote $\boldsymbol{y}=(y_0,\ldots,y_h)$.

\begin{proposition}
  \label{prop:tail-product-cev}
  Let Assumption~\ref{hypo:conditional-indep-measure} hold and assume moreover
  that
  \begin{align}
    \label{eq:condition-moment-product-cev}
    \int_{[0,\infty]^{h+1}} \1{\{y_0y_h>1\}} \nucond_h(\rmd \boldsymbol{y}) <
    \infty \; ,
  \end{align}
  and there exists $\delta>0$ such that
  \begin{align}
    \label{eq:condition-epsilon-product-cev}
    \lim_{\epsilon\to0} \limsup_{x\to\infty}  \frac{\esp \left[ \left(
        \frac{X_0}{x} \1{\{X_0 \leq \epsilon x\}}
        \frac{X_h}{\scalfunccev_h(x)} \right)^\delta \right]}{\pr(X_0>x)} = 0 \; .
\end{align}
Then
\begin{align}
  \label{eq:tail-product-cev}
  \lim_{x\to\infty} \frac{ \pr (X_0X_h > x\scalfunccev_h(x)u) } {\pr(X_0>x)}
  = u^{-\alpha/(1+\kappa_h)}     \int_{[0,\infty]^{h+1}} \1{\{y_0y_h>1\}} \nucond_h(\rmd\boldsymbol{y})  \; .
\end{align}
\end{proposition}
Thus, the right tail index of the product $X_0X_h$ is $\alpha/(1+\kappa_h)$.
\begin{proof}
  Fix some $\epsilon>0$. Then, by vague convergence,
  \begin{align*}
    \lim_{x\to\infty} \frac{\pr(X_0 \1{\{X_0 > \epsilon x\}} X_h >
      x\scalfunccev_h(x) u )}{\pr(X_0>x)} =
    \int_{(\epsilon,\infty]\times[0,\infty]^h} \1{\{y_0y_h>u\}}
    \nucond_h(\rmd\boldsymbol{y}) \; .
  \end{align*}
  and by Markov's inequality,
  \begin{align*}
    \frac{\pr(X_0 \1{\{X_0 \leq \epsilon x\}} X_h > x\scalfunccev_h(x) u )}
    {\pr(X_0>x)} \leq \frac{ \esp \left[ \left( \frac{X_0}{x} \1{\{X_0 \leq
            \epsilon x\}} \frac{X_h}{\scalfunccev_h(x)} \right)^\delta \right]}
    {u^\delta \pr(X_0 > x)} \; .
  \end{align*}
  Conditions~(\ref{eq:condition-moment-product-cev})
  and~(\ref{eq:condition-epsilon-product-cev}) ensure that
  \begin{align*}
    \lim_{x\to\infty} \frac{ \pr (X_0X_h > x\scalfunccev_h(x)u) } {\pr(X_0>x)} =
    \int_{[0,\infty]^{h+1}} \1{\{y_0y_h>u\}} \nucond_h(\rmd\boldsymbol{y}) \; .
  \end{align*}
  This yields~(\ref{eq:tail-product-cev}) by the homogeneity
  property~(\ref{eq:homogeneity-cev}) and the change of variable
  $y_0=u^{1/(1+\kappa_h)} z_0$, $y_i=u^{\kappa_i/(1+\kappa_h)}z_i$, $1 \leq i
  \leq h$.
\end{proof}

\subsection{Convergence of moments}
\label{sec:conv-moment}
The following lemma  states that under suitable moment assumptions, the
convergence~(\ref{eq:lim-cevcdf}) can be extended to unbounded functionals.
\begin{lemma}
  \label{lem:weakconv-unbounded}
  Let Assumption~\ref{hypo:conditional-indep-measure} hold. Assume moreover
  that there exists $x_0>0$ and $q_0,\dots,q_h>0$ such that
  \begin{align}
    \label{eq:condition-weakconv-unbounded1}
    \sup_{x\geq x_0} \esp \left[ \left|\frac{X_0}x\right|^{q_0} \prod_{i=1}^h
      \left| \frac{X_i}{b_i(x)} \right|^{q_i} \mid X_0>x \right] < \infty \; .
  \end{align}
  Let $g$ be a continuous function defined on $[1,\infty)\times\Rset^h$ such
  that
\begin{align}
    \label{eq:condition-weakconv-unbounded2}
  |g(x_0,\dots,x_h)| \leq C\prod_{i=0}^h (|x_i|\vee1)^{q_i'} \; ,
\end{align}
for some $q_i'<q_i$, $0\leq i \leq h$ and a positive constant $C$.  Then
\begin{align}
  \label{eq:convweak-unbounded}
  \lim_{x\to\infty} \esp\left[
    g\left(\frac{X_0}{x},\frac{X_1}{b_1(x)},\dots,\frac{X_h}{\scalfunccev_h(x)}\right)
    \mid X_0>x\right] = \int_1^\infty \int_{\Rset^h} g(\boldsymbol{y})
  \nucond_h(\rmd\boldsymbol{y}) \; .
\end{align}
\end{lemma}
\begin{proof}
  Let $\nucond_{h,x}$ be the measure defined on $(0,\infty)\times\Rset^h$ by
  \begin{align}
   \label{eq:def-numultx}
    \nucond_{h,x}(\cdot) = \frac{1} {\pr(X_0>x)} \pr \left( \left(
        \frac{X_0}x,\frac{X_1}{b_1(x)},\dots,\frac{X_h}{\scalfunccev_h(x)}\right) \in \cdot
    \right) \; .
  \end{align}
  Then, we have
  \begin{align*}
    \esp \left[ g\left( \frac{X_0}x, \frac{X_1}{b_1(x)},
        \dots,\frac{X_h}{\scalfunccev_h(x)} \right) \mid X_0 > x \right] = \int_1^\infty
    \int_{\Rset^h} g(\boldsymbol{y}) \nucond_{h,x}(\rmd \boldsymbol{y}) \; .
  \end{align*}
  Note that $\nucond_{h,x}$ is a probability measure on $[1,\infty)\times\Rset^h$
  which converges weakly to $\nucond_h$. Let $\boldsymbol{Y}_{h,x}$ be a sequence of
  random variables with distribution $\nucond_{h,x}$. Then $\boldsymbol{Y}_{h,x}$ converges
  weakly to a random variable $\boldsymbol{Y}_h$ with distribution
  $\nucond_h$. Therefore, the convergence~(\ref{eq:convweak-unbounded}) holds
  for all bounded and continuous function $g$. If $g$ is unbounded and
  satisfies~(\ref{eq:condition-weakconv-unbounded2}), then
  (\ref{eq:condition-weakconv-unbounded1}) ensures the uniform integrability of
  the sequence $g(\boldsymbol{Y}_{h,x})$ and thus $\lim_{x\to\infty}
  \esp[g(\boldsymbol{Y}_{h,x})]=\esp[g(\boldsymbol{Y}_h)]$.
\end{proof}

\begin{remark}
  \label{rem:conv-moment-indep}
  Condition~(\ref{eq:condition-weakconv-unbounded1}) ensures the uniform
  integrability needed to obtain the convergence of the expectation
  in~(\ref{eq:convweak-unbounded}). In the case of extremal dependence, it is
  necessary that $q_0+\cdots+q_h \leq \alpha$
  for~(\ref{eq:condition-weakconv-unbounded1}) to hold. In the case of extremal
  independence, this is in general neither sufficient nor necessary. For each
  model, the range of the admissible exponents $q_i$ must then be given.
\end{remark}

\subsection{Conditional Tail Expectation}
\label{sec:CTE}
Assumption~\ref{hypo:conditional-indep-measure} and
Lemma~\ref{lem:weakconv-unbounded} can be applied to study certain risk
measures. In a time series context, we may be interested in the limiting
behavior as $x\to\infty$ of the Conditional Tail Expectation (CTE), defined by
\begin{align*}
  \CTE_h(x) = \esp[X_h \mid X_0>x] \; .
\end{align*}
This quantity is related to the expected shortfall (ES), defined by
\begin{align*}
  \ES_h(u) = \esp[X_h \mid X_0>\VAR_{X_0}(u)] \; ,
\end{align*}
where $\VAR_{X_0}(u)$ is the Value-at-Risk associated with the random variable
$X_0$, at the level~$u$.
Note that the expected shortfall (originally defined with $h=0$) is a coherent
risk measure in the sense of \cite{artzner:delbaen:eber:heath:1999}.  The
previous quantities could be zero. In a risk measure context, one might rather
be interested in $\CTE_h^+(x) = \esp[(X_h)_+ \mid X_0>x]$ where $(X_h)_+$
represent the future losses in absolute values.

If for some $h>0$ the vector $(X_0,X_h)$ is extremally dependent and if
$\alpha>1$, then $\CTE_h^+(x)$ will grow linearly with~$x$,
i.e.~$\lim_{x\to\infty} x^{-1} \CTE_h^+(x) > 0$.  For a large class of regularly
varying sequences (e.g.~stationary solutions of stochastic recurrence
equations), all the bivariate marginal distributions of the pairs $(X_0,X_h)$
are extremally dependent. This means that a large value of $X_0$ yields the same
order of magnitude of the $\CTE_h^+$ for {\bf all} lags~$h$. This may not seem
reasonable for many real data sets, e.g.~for high frequency financial data. In
the case of extremal independence, under
Assumption~\ref{hypo:conditional-indep-measure}, if there exists $\epsilon>0$
such that
\begin{align}
  \label{eq:moment-condition-CTE-indep}
  \sup_{x\geq x_0} \frac{ \esp[ |b_h^{-1}(x) X_h|^{1+\epsilon} \1{\{X_0>x\}}]}
  {\pr(X_0>x)} < \infty \; ,
\end{align}
then $\lim_{x\to\infty} x^{-1} \CTE_h^+(x) = 0$. Again, this does not mean that the CTE is
uninformative, but that a smaller normalization is needed in order to obtain a non trivial limit.
Lemma~\ref{lem:weakconv-unbounded} implies that we can define
\begin{align}
    m_h = \int_1^\infty \int_{\Rset^h} (y_h)_+ \, \cevcdf_h(\rmd y_0,\dots,\rmd
    y_h) \; ,   \label{eq:def-mh}
\end{align}
and we have $\CTE_h^+(x) \sim \scalfunccev_h(x) m_h$.

\subsection{The tail process}
\label{sec:tail-process}
In \cite{basrak:segers:2009} the authors define the tail process as the distributional limit of the
sequence $X_0/x,X_1/x,\dots,X_h/x,\dots$ conditionally on $X_0>x$. In the case of extremal
independence, this $X_t/x$ converges weakly to 0 for all $t>0$. Our approach
suggests the following definition which includes the ordinary tail process.
\begin{definition}
  Assume that Assumption \ref{hypo:conditional-indep-measure} holds.  The tail
  process $\{Y_t\}$ is the distributional limit of the sequence
\begin{align*}
\frac{X_0}{x},\frac{X_1}{b_1(x)},\dots,\frac{X_h}{\scalfunccev_h(x)},\dots
\end{align*}
conditionally on $X_0>x$.
\end{definition}
Note that the distribution of $(Y_0,\dots,Y_h)$ is $\cevcdf_h$.  We now give a
representation of the tail process.  Define the measure $G_h$ on $\Rset^h$ by
\begin{align*}
  G_h(y_1,\dots,y_h) = \int_1^\infty \int_{-\infty}^{y_1}
  \cdots\int_{-\infty}^{y_h} \nucond_h(\rmd u_0,u_0^{\kappa_1}\rmd
  u_1,\dots,u_0^{\kappa_h} \rmd u_h) \; .
\end{align*}
Then, using the homogeneity property~(\ref{eq:homogeneity-cev}), we obtain
\begin{align*}
  \int_{y_0}^\infty \int_{-\infty}^{y_1} \cdots\int_{-\infty}^{y_h}
  \nucond_h(\rmd u_0,u_0^{\kappa_1} \rmd u_1,\dots,u_0^{\kappa_h} \rmd u_h) =
  y_0^{-\alpha} G_h(y_1,\dots,y_h) \; .
\end{align*}
Let $(J_1,\dots,J_h)$ be a random vector with distribution $G_h$. Assume that
$(Y_0,\dots,Y_h)$ and $(J_1,\dots,J_h)$ are defined on the same probability
space and such that $Y_0$ and $(J_1,\dots,J_h)$ are independent. Note that
$(J_1,\dots,J_h)$ need not be independent. Then the previous identity yields
that
\begin{align*}
  (Y_0,\dots,Y_h) \eqdistr (Y_0,Y_0^{\kappa_1}J_1,\dots,Y_0^{\kappa_h}J_h) \; .
\end{align*}

We can use the tail process to interpret the moment
condition~(\ref{eq:condition-moment-product-cev}) in
Proposition~\ref{prop:tail-product-cev} which can now be expressed as
\begin{align}
  \label{eq:condition-moment-product-cev-2}
  \esp[J_h^{\alpha/(1+\kappa_h)}] < \infty \; .
\end{align}
The limit distribution of $X_0X_h$ given $X_0>x$ is thus the distribution of
$Y_0^{1+\kappa_h} J_h$. The independence of $J_h$ and $Y_0$ and
Condition~(\ref{eq:condition-epsilon-product-cev}) imply that the tail of
$Y_0^{1+\kappa_h} J_h$ is $\alpha/(1+\kappa_h)$ and that the tail index of
$X_0X_h$ is the same as that of $Y_0^{1+\kappa_h} J_h$.

\subsection{Comparison with Hidden Regular Variation}
\label{sec:hrv}
A different way to quantify the joint extremal behavior of extremally
independent distributions is Hidden Regular Variation (HRV), introduced in
\cite{resnick:2002}.  To simplify the notation, we will discuss hidden regular
variation for non negative random variables only.
Let $\mbc_{h+1}^2$ be the subset of $[0,\infty]^{h+1}$ comprised of vectors with
at least two positive components (denoted $\mathbb{E}^0$ in\cite{resnick:2002});
that is, $\mbc_{h+1}^2$ is $[0,\infty]^{2}$ with the axes removed. For $h=1$,
$\mbc_2^2=(0,\infty]^2$; for $h=2$,
\begin{align*}
  \mbc_3^2 = \left((0,\infty]^2\times [0,\infty]\right) \cup \left(
    (0,\infty]\times[0,\infty]\times(0,\infty]\right) \cup
  \left([0,\infty]\times(0,\infty]^2 \right) \; .
\end{align*}
A vector $(X_0,\dots,X_h)$ satisfying~(\ref{eq:def-mrv}) is said to have Hidden
Regular Variation if there exists a scaling function $d$ such that
$\lim_{s\to\infty} c(s)/d(s)=\infty$ and the sequence of measures $s \pr
(d^{-1}(s) (X_0,\dots,X_h) \in \cdot)$ converges vaguely to a non zero Radon
measure on $\mbc_{h+1}^2$. Under suitable non degeneracy conditions, the
function $d$ must then be regularly varying with some index $\beta\geq\alpha$.
% and the limiting measure has the following homogeneity property:
% $\numult{h}^*(tA)= t^{-\beta} \numult{h}^*(A)$ for any measurable set $A$
% relatively compact in $\mbc_{h+1}^2$, i.e.~there exists $\epsilon>0$ such that
% $(x_0,\dots,x_h)\in A$ implies $\max_{0 \leq i < j \leq h} x_i \wedge x_j
% >\epsilon$ (at least two components greater than $\epsilon$).
HRV implies extremal independence because of the condition $c(s)/d(s)\to\infty$
but HRV does not imply the existence of a limiting conditional distribution. See
e.g.~\cite[Section~6]{heffernan:resnick:2007}
or~\cite[Example~4]{das:resnick:2011}.
Conversely, it is conjectured but not proved in~\cite{das:resnick:2011} that
extremal independence and existence of a conditional limit law implies HRV.

Let us now highlight the differences between these two concepts.  The fundamental theoretical
difference between HRV and CEV lies in the space where vague convergence holds. This difference
entails the following one: HRV only deals with joint exceedances of two components of the vector,
and in dimension greater than two, the hidden exponent measure may be concentrated on hyperplanes;
this is the case for instance of a vector with three i.i.d.~regularly varying components. The CEV
Assumption~\ref{hypo:conditional-indep-measure} prevents such a degeneracy.  For example, if $X,Y,Z$
are i.i.d.~non negative regularly varying random variables with tail index $\alpha$, then $(X,Y,Z)$
has HRV with $\beta=2\alpha$ but, for $u,v,w>0$,
\begin{align*}
  \lim_{x\to\infty} \frac{ \pr(X>xu,Y>xv,Z>xw) }{\pr^2(X>x)} & = 0 \: , \\
  \lim_{x\to\infty} \frac{ \pr(X>xu,Y>v,Z>w) }{\pr(X>x)} & = u^{-\alpha}
  \pr(Y>v)\pr(Z>w) \: .
\end{align*}
Therefore HRV is uninformative for such exceedances but CEV yields a non
degenerate limit.

The CEV assumption is also more flexible than HRV since it can also accommodate
extremally dependent vectors (ruled out by HRV) with extremally independent
subvectors. For example, we already seen that if $X$ and $Y$ are i.i.d.~non
negative regularly varying random variables with tail index $\alpha$, then
$(X,X,Y)$ is extremally dependent; however, for~$u,u',v>0$,
\begin{align*}
  \lim_{x\to\infty} \frac{ \pr(X>xu,X>xu',Y>xv) }{\pr(X>x)} & = 0 \: , \\
  \lim_{x\to\infty} \frac{ \pr(X>xu,X>xu',Y>v) }{\pr(X>x)} & = (u\vee
  u')^{-\alpha} \pr(Y>v) \: .
\end{align*}
The first limit is obtained by the standard multivariate regular variation
property and is zero even though the vector is extremally dependent, and the
second one shows that the CEV assumption is fulfilled.

In conclusion, we can say that the practical purposes of HRV and CEV are
different: HRV gives an approximation of the probability of exceedances of pairs
of components of an extremally independent vector over a ``not too extreme''
level, whereas CEV quantifies the influence of an extreme component (an extreme
event at time zero in a time series context) on all other components (the future
observations), be they extremally independent or dependent.  There is no
implication or exclusion between HRV and CEV and one approach cannot be deemed
superior to the other. In higher dimensions, the CEV approach seems more
flexible.

\subsection{A counter example}
\label{sec:counter}
We now give an example, where the conditional laws do not exists. Consider a
stationary standard Gaussian process $\sequence{\xi}{t}{\Nset}$ and define $X_t
= \rme^{c\xi_t^2}$, with $c<1/2$. Assume moreover that $|\cov(\xi_0,\xi_n)|<1$
for all $n\geq1$. This is not a stringent assumption since a sufficient
condition is that the process $\{\xi_t\}$ has a spectral density $f$ such that
$\int_{-\pi}^\pi f(t) \,\rmd t=1$. In that case, extremal independence holds for
the bivariate distributions, but a non trivial limiting conditional distribution
$\rme^{c\xi_h^2}$ given $\rme^{c\xi_0^2} > x$ does not exist.  See
\cite[section~2.4]{heffernan:resnick:2007}.
% On the other hand, the tail dependence index of $(X_0,X_h)$ is $(1+\rho_h)/(2-\rho_h)$, with
% $\rho_h=\cov(\xi_0,\xi_h)$.

\section{Extremally independent Markov chains}
\label{sec:markov-chains}

The extremal properties of Markov chains have received considerable attention
recently; see \cite{janssen:segers:2013}, \cite{resnick:zeber:2013:asy} and the
references therein.  The aforementioned papers deal with extremal dependence. In
this section, we will extend some results of \cite{resnick:zeber:2013:asy} to
the present context which allows for extremal independence.
Since the distribution of a Markov chain is entirely determined by its initial distribution and the
transition kernel, denoted by $\Pi$, it is natural in this context to replace
Assumption~\ref{hypo:conditional-indep-measure} by the following one, which is similar
to~\cite[Assumption~2.5]{resnick:zeber:2013:asy}.  For simplicity, we assume that the state space
is~$[0,\infty)$.

\begin{hypothesis}
  \label{hypo:markov-kernel}
  There exist a function $b$, regularly varying at infinity with index $\kappa\geq0$
  and a distribution function $G$ on $[0,\infty)$, not concentrated on one point such
  that
\begin{align}
  \label{eq:kernel}
  \lim_{x\to\infty} \Pi(x,b(x)A) = G(A)
\end{align}
for all Borel sets $A \subset [0,\infty)$ such that $G(\partial A)=0$.
\end{hypothesis}
This means that the transition kernel is asymptotically homogeneous. It also
means that conditionally on $X_0=x$, the distribution of $X_1/b(x)$ converges
weakly to the distribution~$G$.

The main result of this section states that Assumption~\ref{hypo:markov-kernel}
implies Assumption~\ref{hypo:conditional-indep-measure}.  Define $b_0(x)=x$,
$b_1(x)=b(x)$ and for $h\geq1$, $\scalfunccev_h = b_{h-1}\circ b$.
\begin{theorem}
  \label{theo:markov-indep}
  Let $\{X_t\}$ be a Markov chain whose transition kernel satisfies
  Assumption~\ref{hypo:markov-kernel} and with initial distribution having right
  tail index $\alpha>0$.  Assume moreover that~\mbox{$G(\{0\})=0$}. Then
  Assumption~\ref{hypo:conditional-indep-measure} holds and the limiting
  conditional distribution of
  \begin{align*}
    \left(\frac{X_0}x,\frac{X_1}{b_1(x)},
      \dots,\frac{X_h}{\scalfunccev_h(x)},\dots\right)
  \end{align*}
  given $X_0>x$ when $x\to\infty$ is the distribution of the exponential AR(1)
  process $\{Y_t,t\geq0\}$ defined by $Y_t=Y_{t-1}^\kappa W_t$ where $\{W_t\}$
  is an i.i.d. sequence with distribution $G$, independent of the standard Pareto
  random variable $Y_0$ with tail index~$\alpha$.
\end{theorem}

The proof is in Section~\ref{sec:prooftheomarkov}.  For a Markov chain, the tail
process is called the tail chain.  With the
normalization used here, we obtain a new type of tail chain which is an
exponential AR(1) process. In the case of extremal dependence, the usual tail
chain is an exponential random walk. This corresponds to the case $\kappa_j=1$
for all $j$.

Since a Markov chain $\{X_t\}$ can always be expressed as $X_{t+1}=\Phi(X_t,\epsilon_{t+1})$, where
$\{\epsilon_t,t\geq1\}$ is an i.i.d.~sequence (the innovations), independent of $X_0$,
Condition~(\ref{eq:kernel}) is equivalent to the weak convergence of $b^{-1}(x)\Phi(x,\epsilon_0)$
to the distribution $G$ in~(\ref{eq:kernel}). This is the framework considered in
\cite{janssen:segers:2013} under the additional assumption that
$b_j(x)=x$ for all $j$.

If $G(\{0\})>0$, then Theorem~\ref{theo:markov-indep} may no longer be
true. However, without this condition, it can be seen from the proof that the
convergence still holds provided $X_1,\dots,X_{h-1}$ are separated from zero,
i.e.
\begin{multline}
  \label{eq:markov-tail-expar-truncated}
  \lim_{x\to\infty} \pr \left( \frac{X_0}x \leq y_0, \epsilon \leq
    \frac{X_1}{b_1(x)} \leq y_1, \dots,\epsilon \leq \frac{X_{h-1}}{b_{h-1}(x)}
    \leq y_{h-1}, \frac{X_h}{b_h(x)} \leq y_h \mid X_0 > x \right) \\
  = \pr(Y_0 \leq y_0, \epsilon \leq Y_1 \leq y_1,\dots, \epsilon \leq Y_{h-1}
  \leq y_{h-1}, Y_h \leq y_h) \;,
\end{multline}
with $\{Y_j\}$ is the tail process as described in Theorem~\ref{theo:markov-indep}.

In the extremally dependent case (where $b_j(x)=x$ for all $x$), the convergence
still holds under an additional regularity condition, see
\cite[Proposition~5.1]{resnick:zeber:2013:asy}. It would be possible to
generalize this condition to the extremally independent context, but as in the
extremally dependent case, this condition would not be necessary for the
convergence to holds. Therefore we do not pursue in this direction.  Note
finally that if $G(\{0\})>0$ then the tail process is identically zero after a
geometric time with mean $1/G(\{0\})$.

We now give examples of Markov chains satisfying condition~(\ref{eq:kernel}).
\subsection{Exponential AR(1)}
\label{sec:expar}
Let the time series $\{V_t\}$ be defined by
$V_{t}=\rme^{\xi_{t}}$ with
\begin{align}
  \label{eq:ar1}
  \xi_{t} = \phi \xi_{t-1}+\epsilon_{t} \; ,
\end{align}
where $0 \leq \phi<1$ and $\sequence{\epsilon}{t}{\Zset}$ is an i.i.d.~sequence
such that $\esp[\epsilon_0] = 0$ and
\begin{align}
  \label{eq:expepsilon-rv}
  \pr(\rme^{\epsilon_0} > x) = x^{-\alpha}\ell(x)  \; ,
\end{align}
for some $\alpha>0$ and a slowly varying function $\ell$. In other words,
$\rme^{\epsilon_0}$ has a regularly varying right tail with index $\alpha$.
\cite[Section~3]{mikosch:rezapour:2013} studied the regular variation and proved the
extremal independence of this model.  Let $\xi_{t} = \sum_{j=0}^\infty \phi^j
\epsilon_{t-j}$ be the stationary solution of the AR(1) equation~(\ref{eq:ar1}).
Condition~(\ref{eq:expepsilon-rv}) implies
that for $\phi\in[0,1)$ and  $j\geq1$,
\begin{align*}
  \esp[\rme^{\alpha\phi^j \epsilon_0}] < \infty \; .
\end{align*}
Thus, applying Breiman's lemma, we have
\begin{align*}
  \pr(V_t>x)= \pr(\rme^{\xi_t}>x) & = \pr(\rme^{\epsilon_t}
  \rme^{\sum_{j=1}^\infty \phi^j \epsilon_{t-j}}>x) \\
  & \sim \pr(\rme^{\epsilon_0}>x) \esp \left[ \rme^{\alpha\sum_{j=1}^\infty
      \phi^j \epsilon_{t-j}} \right] =  \pr(\rme^{\epsilon_0}>x)  \prod_{j=1}^\infty \esp\left[
    \rme^{\alpha \phi^j \epsilon_{t-j}} \right] \; .
\end{align*}
That is, $V_t$ has a regularly varying right tail and is tail equivalent to
$\rme^{\epsilon_0}$.  The Exponential AR(1) satisfies the equation
\begin{align}
  \label{eq:expar-recurrence}
  V_{t+1}=V_t^{\phi}\rme^{\epsilon_{t+1}} \; .
\end{align}
This corresponds to the functional representation $\Phi(x,\epsilon) =
x^\phi\epsilon$.  We have
\begin{align*}
  \Pi(x,A) = \pr (\rme^{\epsilon_0} \in x^{-\phi}A) \; ,
\end{align*}
and thus, with $G$ the distribution of $\rme^{\epsilon_0}$, we have
\begin{align*}
  \Pi(x,x^\phi A) = G(A) \; .
\end{align*}
Since $G(\{0\})=0$, Theorem \ref{theo:markov-indep} is applicable. The tail
chain is a non stationary exponential AR(1) process $\{Y_t\}$ defined by
$Y_t=Y_{t-1}^\phi \rme^{\epsilon_t}$ and $Y_0$ is a standard Pareto random
variable.

For
$(y_0,\dots,y_h)\in[1,\infty)\times\Rset^h$, we have
\begin{multline}
  \label{eq:jointlimdist-expar1}
  \lim_{x\to\infty}
  \pr( V_0 \leq xy_0,V_1  \leq x^{\phi} y_1,\dots,V_h \leq  x^{\phi^h} y_h \mid V_0 > x) \\
  = \int_1^{y_0} \pr( \rme^{\xi_{0,1}} \leq v^{-\phi} y_1, \dots,
  \rme^{\xi_{0,h}} \leq v^{-\phi^h} y_h ) \, \alpha v^{-\alpha-1} \rmd v \; ,
\end{multline}
where for $h\geq1$, $\xi_{0,h}=\sum_{j=0}^{h-1}\phi^j \epsilon_{h-j}$.  The
limiting conditional distribution of $V_h$ given $V_0>x$ is thus
\begin{align*}
  \lim_{x\to\infty} \pr( V_h \leq x^{\phi^h} y \mid V_0 > x) = \int_{1}^\infty
  \pr( \rme^{\xi_{0,h}} \leq v^{-\phi^h} y ) \, \alpha v^{-\alpha-1} \rmd v \; .
\end{align*}
The conditional scaling exponent is $\kappa_h=\phi^h$.  We also note
that since $\kappa_h\in(0,1)$, this distribution is tail equivalent
to the distribution of $\rme^{\epsilon_0}$, i.e.~we have as $y\to\infty$,
\begin{align*}
  \pr( V_h > x^{\phi^h} y \mid V_0 > x) & = \int_1^\infty \pr( \rme^{\xi_{0,h}}
  > v^{-\kappa_h} y ) \, \alpha v^{-\alpha-1} \rmd v \\
  & \sim \pr(\rme^{\xi_{0,h}} >y) \int_1^\infty v^{\alpha\kappa_h} \alpha
  v^{-\alpha-1} \, \rmd v \\
  & = \frac{ \pr (\rme^{\xi_{0,h}} >y) } {1-\kappa_h} \sim \frac{\pr
    (\rme^{\xi_{0}} >y)} {(1-\kappa_h)\esp[\rme^{\alpha \kappa_h \xi_0}] } \; .
\end{align*}
If $\alpha>1$, we can apply Lemma~\ref{lem:weakconv-unbounded} and we
obtain
\begin{align*}
  \lim_{x\to\infty} \esp\left[ \frac{V_h}{x^{\kappa_h}} \mid V_0>x \right] =
  \frac{\alpha\esp[\rme^{\xi_{0,h}}]} {\alpha-\kappa_h} = \frac{\alpha\esp[V_0]}
  {(\alpha-\kappa_h)\esp[V_0^{\kappa_h}]} \; .
\end{align*}

\paragraph{Tail of $V_0V_h$.}
The recursion~(\ref{eq:ar1}) yields $V_0V_h = V_0^{1+\phi^h} \rme^{\sum_{j=0}^{h-1} \phi^j
  \epsilon_h-j}$ and the series in the exponential is independent of $V_0$. Also,
$\esp[\rme^{\alpha(1+\phi^h)^{-1}\sum_{j=0}^{h-1} \phi^j \epsilon_h-j}]<\infty$ and by Breiman's
Lemma, we obtain directly that the tail index of $V_0V_h$ is $\alpha/(1+\phi^h)$. We can also check
that Condition~(\ref{eq:condition-epsilon-product-cev}) holds.  Fix $\delta<\alpha$ such that
$\delta(1+\phi^h)>\alpha$ and define $\scalfunccev_h(x)=x^{\phi^h}$. Then, for some constant $C>0$,
\begin{align*}
  \esp \left[ \left( \frac{V_0}{x} \1{\{V_0 \leq \epsilon x\}} \frac{V_h}{\scalfunccev_h(x)}
    \right)^\delta \right] = \esp\left[ \frac{V_0^{\delta(1+\phi^h)}}{x^{\delta(1+\phi^h)}} \1{\{V_0
      \leq \epsilon x\}} \right] \esp \left[ \rme^{\delta\sum_{j=0}^{h-1}
  \phi^j \epsilon_h-j} \right] \leq C \epsilon^{\delta(1+\phi^h)}
  \pr(\rme^{\epsilon_0}>\epsilon x) \; .
\end{align*}
This yields
\begin{align*}
  \limsup_{x\to\infty} \frac{1}{\pr (\rme^{\epsilon_0}>x)} \esp \left[ \left( \frac{V_0}{x} \1{\{V_0 \leq
        \epsilon x\}} \frac{V_h}{\scalfunccev_h(x)} \right)^\delta \right] \leq C \epsilon^{\delta(1+\phi^h)-\alpha}\; .
\end{align*}
By the choice of $\delta$, this yields the negligibility
condition~(\ref{eq:condition-epsilon-product-cev}).

\paragraph{Convergence of moments} Using the same decomposition as above, we obtain that the moment
condition~(\ref{eq:condition-weakconv-unbounded1}) holds if, for $i=0,\dots,h$,
\begin{align*}
  \sum_{j=0}^i \phi^j q_{h-j} < \alpha \; .
\end{align*}
This implies in particular that $q_j<\alpha$ for all $j=0,\dots,h$. 
\paragraph{Explosive case.}
Consider now the case $\phi>1$. If the exponential AR(1) model is defined by the
recurrence equation (\ref{eq:expar-recurrence}): $V_{t+1} =
V_t^\phi\rme^{\epsilon_{t+1}}$ with $\{\epsilon_t,t\geq1\}$ independent of
$V_0$, then the limit (\ref{eq:jointlimdist-expar1}) still holds, but a
stationary measure for this Markov chain does not exist.

On the other hand, the stationary (non-causal and non-Markovian) solution of
Equation~(\ref{eq:ar1}) is given by $\xi_t = \sum_{j=1}^\infty (-\phi)^{-j}
\epsilon_{t+j}$. Then the sequence $\{\rme^{\xi_t}\}$ is stationary and
regularly varying but the tail index is now $\alpha\phi$ and no conditional
limiting distribution exist.

Finally, it is obvious that the time-reversed chain has an invariant measure and
the limiting conditional distribution $\pr(V_0 < x^{1/\phi}y_1\mid V_1>x)$
exists.

\subsection{Switching exponential AR(1)}
\label{sec:switchingAR}

Let $\{U_t\}$ be an i.i.d. sequence with uniform marginal distribution on
$[0,1]$, and let $\{R_t\}$ be an i.i.d. sequence with marginal distribution
$F_R$ concentrated on $[0,\infty)$, independent of the sequence $\{U_t\}$. Let
$\phi>0$, $k:[0,\infty)\to[0,1]$ be a measurable function and define a Markov
chain $\{X_t\}$ by $X_0$ and
\begin{align*}
  X_{t+1} = R_{t+1} ( X_t^\phi \1{\{k(X_t)\leq U_{t+1}\}} + \1{\{k(X_t) > U_{t+1}\}} ) \; .
\end{align*}
This is a multiplicative version of the Stochastic Unit Root process; see
\cite{gourieroux:robert:2006}.  The transition kernel $\Pi$ of the chain is defined by
\begin{align*}
  \Pi(x,A) = F_R(x^{-\phi}A) (1-k(x)) + F_R(A)k(x) \; .
\end{align*}
If \mbox{$\lim_{x\to\infty} k(x)=\eta$}, then Condition~(\ref{eq:kernel}) holds
with $G$ defined by
\begin{align*}
  G(A) = \lim_{x\to\infty} \Pi(x,x^\phi A) = F_R(A) (1-\eta) + \eta\delta_0(A)\; ,
\end{align*}
where $\delta_0$ is the Dirac mass at 0. If $\eta=0$, then we can apply
Theorem~\ref{theo:markov-indep}. The conditional scaling exponent at lag 1 is
$\kappa_1=\phi$. If $\eta>0$, then $G(\{0\})>0$ and Theorem
\ref{theo:markov-indep} is not applicable.  However, in the simple case
$k(x)\equiv \eta$, it is readily checked that the conclusion of
Theorem~\ref{theo:markov-indep} nevertheless holds, i.e.~if the distribution of
$X_0$ has a right tail index $\alpha>0$, then, conditionally on $X_0>x$,
$(x^{-1}X_0, x^{-\phi} X_1,\dots,x^{-\phi^h} X_h)$ converges to the tail process
$(Y_0,Y_1,\dots,Y_h)$ where $Y_j=Y_{j-1}^\phi W_j$, $j\geq1$, $Y_0$ has a standard Pareto
distribution with tail index $\alpha$ and $W_1,\dots,W_h$ are i.i.d.~with
distribution $G$.

Let us now briefly discuss the existence of a stationary distribution for this
Markov chain. We apply the Foster-Lyapunov criterion. See
\cite{meyn:tweedie:2009}.  Define $V(x)=\log(x)$ and $c=\esp[\log(R_0)]$. Then
we have $\Pi V(x) = \phi\{1-k(x)\}V(x) + c$. Assume that $R_0$ has an absolutely
continuous distribution with a positive density around~0 so that the chain is
irreducible. If $\phi\in(0,1)$, then $\Pi V(x) \leq \phi V(x) + c$, and thus
there exists a unique invariant distribution and the chain is geometrically
ergodic.

\section{Exponential linear process}
\label{sec:explin}
The result for the Exponential AR(1) model can be extended to a (non-Markovian)
exponential linear process $\rme^{\xi_t}$ with possible long memory.

%  This allows for long memory in the sense that it is not required
% that $\sum_{j=1}^\infty \phi_{j}<\infty$.  We next prove that
% Assumption~\ref{hypo:conditional-indep-measure} holds.

\begin{lemma}
  \label{lem:linear_assumption-1}
  Define $\xi_t = \sum_{j=0}^\infty \phi_{j} \epsilon_{t-j}$, where $\{\epsilon_{t}\}$ is a sequence
  of i.i.d.~random variables such that $\esp[\epsilon_0]=0$ and such that \eqref{eq:expepsilon-rv}
  holds, $\phi_0=1$, $\sum_{j=1}^\infty \phi_{j}^2<\infty$ and $0 \leq \phi_{j} <1$ for
  all~$j\geq1$.

  The sequence of measures defined on $(0,\infty]\times[0,\infty]^h$ by
  \begin{align*}
  \frac1{\pr(\rme^{\xi_0}>x)} \pr \left( \left(\frac{\rme^{\xi_0}}{x},
      \frac{\rme^{\xi_1}}{\scalfunccev_{1}(x)}, \cdots,
      \frac{\rme^{\xi_h}}{\scalfunccev_{h}(x)} \right) \in \cdot \right)
\end{align*}
converges vaguely to the measure $\nucond_{h}$ defined by
  \begin{align*}
    \nucond_{h}(\cdot) = \frac1{\esp[\rme^{\alpha\xi_0^*}]} \int_0^\infty \pr\left(
      (v\rme^{\xi_0^*},v^{\phi_1}\rme^{\xi_1^*}\dots,v^{\phi_h}\rme^{\xi_h^*})
      \in \cdot \right) \alpha v^{-\alpha-1} \, \rmd v \; .
  \end{align*}
  where $\xi_{j}^* = \xi_j - \phi_j\epsilon_0$ and $b_j(x)=x^{\phi_j}$, $j\geq1$.
\end{lemma}

Equivalently, we obtain the following limiting conditional distribution.  For
$y_0\geq1$ and~$(y_0,y_1,\dots,y_h)\in(0,\infty)\times\Rset^h$,
  \begin{multline}
    \label{eq:formula_linear_multivariate}
    \lim_{x\to\infty} \pr(\rme^{\xi_0} > xy_0,\rme^{\xi_1}\leq x^{\phi_1}
    y_1,\dots,\rme^{\xi_h} \leq x^{\phi_h} y_h \mid \rme^{\xi_0} >x) \\
    = \frac1{\esp[\rme^{\alpha\xi_0^*}]} \int_0^\infty \pr\left(\rme^{\xi_0^*} >
      v y_0, \rme^{\xi_1^*} \leq v^{\phi_1} y_1, \dots, v^{\phi_h}\rme^{\xi_h^*}
      \leq y_h \right) \, \alpha \, v^{-\alpha-1} \, \rmd v \; .
  \end{multline}
Thus, the lag $h$ conditional scaling index is $\phi_h$.  If the coefficients
$\phi_j$ are decreasing, i.e. $\phi_j>\phi_{j+1}$ for all $j\geq0$, then the
index of hidden regular variation of $(\rme^{\xi_0},\rme^{\xi_j})$ is
$\alpha(2-\phi_j)$. Otherwise, it is the solution of an infinite dimensional
optimization problem and may take any value. See \cite{janssen:drees:2013} for
more details.
\begin{proof}[Proof of Lemma~\ref{lem:linear_assumption-1}]
  To avoid trivialities, we assume that $\phi_j\ne0$ for at least one index $j\geq1$.  By
  definition, we have $\xi_k = \phi_k \epsilon_0 + \xi_{k}^*$ for all $k\geq0$. Note that
  $\epsilon_0$ is independent of $\xi_{k}^*$, $k\ge 0$. Denote the distribution of
  $\rme^{\epsilon_0}$ by $F_\epsilon$ and define the measure $\sigma_x$ by $\sigma_x(\rmd v) =
  F_\epsilon(x\rmd v)/\bar F_\epsilon(x)$.  The measure $\sigma_x$ converges vaguely on $(0,\infty]$
  to the measure with density $\alpha v^{-\alpha-1} \, \rmd v$; see
  \cite[Theorem~3.6]{resnick:2007}.  Then, for $(y_0,\dots,y_h)\in(0,\infty)\times[0,\infty]^h$,
\begin{align*}
  \pr(\rme^{\xi_0} & > xy_0, \rme^{\xi_1} \leq x^{\phi_1} y_1,\dots,\rme^{\xi_h}
  \leq x^{\phi_h} y_h)  \\
  & = \pr(\rme^{\epsilon_0+\xi_{0}^*} > xy_0, \rme^{\phi_1\epsilon_0}
  \rme^{\xi_{1}^*} \leq  x^{\phi_1} y_1,\dots, \rme^{\phi_h \epsilon_0} \rme^{\xi_{h}^*} \leq x^{\phi_h} y_h ) \\
  & = \int_{u=0}^\infty \pr(\rme^{\xi_{0}^*}>(x/u)y_0, \rme^{\xi_{1}^*} \leq
  (x/u)^{\phi_1} y_1,\dots,  \rme^{\xi_{h}^*} \leq  (x/u)^{\phi_h} y_h ) F_\epsilon(\rmd u) \\
  & = \bar F_\epsilon(x) \int_{v=0}^\infty \pr(\rme^{\xi_{0}^*}>v^{-1}y_0,
  \rme^{\xi_{1}^*} \leq v^{-\phi_1} y_1,\dots, \rme^{\xi_{h}^*} \leq
  v^{-\phi_h} y_h ) \, \sigma_x(\rmd v) \; .
\end{align*}
In order to prove the convergence of the integral, we must split it into two
parts. Define the function $\tilde{K}_h$ on $(0,\infty)^2\times\Rset^h$ by
\begin{align*}
  \tilde{K}_h(v,y_0,\dots,y_h) = \pr(\rme^{\xi_{0}^*}>v^{-1}y_0,
  \rme^{\xi_{1}^*} \leq v^{-\phi_1} y_1,\dots, \rme^{\xi_{h}^*} \leq v^{-\phi_h}  y_h )\; .
\end{align*}The function $\tilde{K}_h$ is uniformly bounded (by
one), thus for $c>0$, we have,
\begin{align*}
  \lim_{x\to\infty} \int_c^\infty \tilde{K}_h(v,y_0,y_1,\dots,y_h) \,
  \sigma_x(\rmd v) = \int_c^\infty \tilde{K}_h(v,y_0,y_1,\dots,y_h) \, \alpha \,
  v^{-\alpha-1} \, \rmd v \; .
\end{align*}
Let $\phi^*=\sup_{j\geq 1}\phi_j$. By assumption, $0<\phi^*<1$. Thus $\rme^{\xi_{0}^*}$ has the tail
index $\alpha/\phi^*>\alpha$.  Moreover, since $\epsilon_0$ is independent of $\xi_{0}^*$, by Markov's
inequality, we have, for $\alpha<q<\alpha/\phi^*$,
\begin{align*}
  \int_0^c \tilde{K}_h(v,y_0,y_1,\dots,y_h) \, \sigma_x(\rmd v) & \leq
  \frac{\pr(\rme^{\epsilon_0}\rme^{\xi_{0}^*}>xy_0,\rme^{\epsilon_0}\leq cx)}{\bar{F}_\epsilon(x)} \\
  & \leq \frac{\esp[\rme^{q\xi_{0}^*}] \, \esp[\rme^{q\epsilon_0}
    \1{\{\rme^{\epsilon_0} \leq cx\}}]} {(xy_0)^q\bar{F}_\epsilon(x)} \; .
\end{align*}
We obtain that
$$
\limsup_{x\to\infty} \int_0^c \tilde{K}_h(v,y_0,y_1,\dots,y_h) \, \sigma_x(\rmd v) = O(c^{q-\alpha})
$$
and thus since $q>\alpha$,
\begin{align*}
  \lim_{c\to0} \limsup_{x\to\infty} \int_0^c \tilde{K}_h(v,y_0,y_1,\dots,y_h) \,
  \sigma_x(\rmd v) = 0 \; .
\end{align*}
We may now conclude that
\begin{align*}
  \lim_{x\to\infty} \int_0^\infty \tilde{K}_h(v,y_0,y_1,\dots,y_h) \,
  \sigma_x(\rmd v) = \int_0^\infty \tilde{K}_h(v,y_0,y_1,\dots,y_h) \, \alpha \,
  v^{-\alpha-1} \, \rmd v \; .
\end{align*}
Since $\lim_{x\to\infty} \pr(\rme^{\xi_0}>x)/\pr(\rme^{\epsilon_0}>x) =
\esp[\rme^{\alpha\xi_{0}^*}]$, we finally obtain
\begin{multline*}
  \lim_{x\to\infty} \pr(\rme^{\xi_0} > xy_0, \rme^{\xi_1} \leq x^{\phi_1} y_1,
  \dots, \rme^{\xi_h} \leq x^{\phi_h} y_h \mid \rme^{\xi_0} > x)  \\
  = \frac1{\esp[\rme^{\alpha\xi_{0}^*}]} \; \int_0^\infty
  \tilde{K}_h(v,y_0,y_1,\dots,y_h) \, \alpha \, v^{-\alpha-1} \, \rmd v \; ,
\end{multline*}
which is exactly~(\ref{eq:formula_linear_multivariate}).
\end{proof}

The AR(1) process is a particular case of a linear process with $\phi_j=\phi^j$
for all $j\geq0$, so (\ref{eq:formula_linear_multivariate}) and
(\ref{eq:jointlimdist-expar1}) must coincide in this case. To check this,
recalling the notation of Section~\ref{sec:expar}, note that for the exponential
AR(1) we have $\xi_{k}^*=\xi_{0,k}+\phi^k\xi_{0}^*$ and $\xi_{0}^*$ is
independent of $\xi_{0,k}$ for each $k\ge 1$. Denoting by $F_{*}$ the distribution
of $\rme^{\xi_{0}^*}$ and considering for clarity only the case  $h=1$, we have
\begin{align*}
  \frac1{\esp[\rme^{\alpha\xi_{0}^*}]} & \; \int_0^\infty
  \pr(\rme^{\xi_{0}^*}>v^{-1}y_0,\rme^{\xi_{1}^*}\le v^{-\phi}y_1) \, \alpha \, v^{-\alpha-1} \, \rmd v \\
  & = \frac1{\esp[\rme^{\alpha\xi_{0}^*}]} \;
  \int_{v=0}^\infty\int_{s=v^{-1}y_0}^{\infty}
  \pr(\rme^{\xi_{0,1}}\le (sv)^{-\phi}y_1)\, F_{*}(\rmd s) \, \alpha \, v^{-\alpha-1} \, \rmd v \\
  & = \frac1{\esp[\rme^{\alpha\xi_{0}^*}]} \;
  \int_{s=0}^\infty\int_{v=s^{-1}y_0}^{\infty}
  \pr(\rme^{\xi_{0,1}}\le (sv)^{-\phi}y_1)\,  \alpha \, v^{-\alpha-1} \, \rmd v \,F_{*}(\rmd s)  \\
  & = \frac1{\esp[\rme^{\alpha\xi_{0}^*}]} \; \int_{s=0}^\infty s^{\alpha}
  F_{*}(\rmd s) \int_{y_0}^{\infty} \pr(\rme^{\xi_{0,1}}\le u^{-\phi}y_1)\,
  \alpha \, u^{-\alpha-1} \, \rmd u  \\
  & = \int_{y_0}^{\infty} \pr(\rme^{\xi_{0,1}}\le u^{-\phi}y_1)\, \alpha \,
  u^{-\alpha-1} \, \rmd u\; .
\end{align*}
That is, with $h=1$, Equation (\ref{eq:formula_linear_multivariate}) reduces to
(\ref{eq:jointlimdist-expar1}).

\section{Stochastic volatility models}
\label{sec:sv}
\subsection{Stochastic volatility process with heavy tailed volatility}
\label{sec:heavytailedvol}
Assume now as in~\cite{mikosch:rezapour:2013} that $X_t =
V_tZ_t=\rme^{\xi_t}Z_t$, where $\sequence{\xi}{t}{\Zset}$ is the AR(1) process
considered in Section \ref{sec:expar} and $\sequence{Z}{t}{\Zset}$ is a sequence
of i.i.d.~random variables such that $\esp[|Z_0|^{q}]<\infty$ for some
$q>\alpha$, independent of the sequence $\sequence{\xi}{t}{\Zset}$. Breiman's
lemma yields, for $x\to\infty$,
\begin{align}
  \label{eq:tail-SV-MR}
  \pr(X_0>x) & \sim \esp[(Z_0)_{+}^{\alpha}] \pr(\rme^{\xi_0}>x) \; , \\
  \pr(X_0<-x) & \sim \esp[(Z_0)_{-}^{\alpha}] \pr(\rme^{\xi_0}>x) \; .
\end{align}
Let $F_0$ be the distribution function of $X_0$ and $\nu_x$ is the
measure defined by $\nu_x(\rmd u)=F_0(x\rmd u)/\bar F_0(x)$. Then,
conditioning on $X_0$ and the sequence $\{Z_t\}$, we
have, for $y_0\geq1$ and $(y_1,\dots,y_h)\in\Rset^h$,
\begin{align*}
  \pr(X_0 & > xy_0, X_1 \leq x^{\phi_1} y_1, \dots, X_h \leq x^{\phi_h} y_h \mid X_0 > x) \\
  & = \frac{\pr(Z_0V_0 > x y_0, Z_1\rme^{\xi_{0,1}} V_0^\phi \leq
    x^\phi y_1,\dots, Z_h\rme^{\xi_{0,h}} V_0^{\phi^h} \leq x^{\phi^h} y_h)}{\pr(X_0>x)} \\
  & = \frac{1}{\pr(X_0>x)} \int_0^\infty \pr(Z_0 > (x/u) y_0,
  Z_1\rme^{\xi_{0,1}} \leq (x/u)^\phi y_1,\dots, Z_h\rme^{\xi_{0,h}} \leq
  (x/u)^{\phi^h} y_h) \, F_0(\rmd u) \\
  & = \frac{\pr(V_0>x)}{\pr(X_0>x)} \int_0^\infty \pr(Z_0 > v^{-1} y_0,
  Z_1\rme^{\xi_{0,1}} \leq v^{-\phi} y_1,\dots, Z_h\rme^{\xi_{0,h}} \leq
  v^{-\phi^h} y_h) \, \nu_x(\rmd v) \; .
\end{align*}

By arguments similar to those in the proof
of~(\ref{eq:formula_linear_multivariate}), we obtain
\begin{multline*}
  \lim_{x\to\infty}   \pr(X_0  > xy_0, X_1 \leq x^{\phi_1} y_1, \dots, X_h \leq x^{\phi_h} y_h \mid X_0 > x) \\
  = \frac1{\esp[(Z_0)_+^\alpha]} \; \int_0^\infty \pr(Z_0 > v^{-1} y_0,
  Z_1\rme^{\xi_{0,1}} \leq v^{-\phi} y_1,\dots, Z_h\rme^{\xi_{0,h}} \leq
  v^{-\phi^h} y_h) \alpha \, v^{-\alpha-1} \, \rmd v \; .
\end{multline*}

The conditional scaling exponent is thus $\kappa_h = \phi^h$ and the
limiting conditional distribution of $x^{-\kappa_h}X_h$ given
$X_0>x$ is
\begin{align*}
  \lim_{x\to\infty} \pr(X_h \leq x^{\kappa_h} y \mid X_0>x) = \int_0^\infty
  \pr(Z_0>v^{-1} \, , \; Z_h\rme^{\xi_{0,h}} \leq v^{-\kappa_h}y)  \alpha  v^{-\alpha-1}  \, \rmd v \; .
\end{align*}
If $\alpha>1$, we can apply Lemma~\ref{lem:weakconv-unbounded}  to obtain
\begin{align*}
  \lim_{x\to\infty} \esp \left[ \frac{(X_h)_{+}}{x^{\kappa_h}} \mid X_0 > x \right]
  & = \frac{ \alpha \, \esp \left [ (Z_0)_+^{\alpha-\kappa_h} \right] \, \esp[(Z_0)_+] \,
    \esp [ X_0 ] } {(\alpha-\kappa_h)\esp[(Z_0)_+^\alpha] \esp[X_0^{\kappa_h}]} \; .
\end{align*}
\paragraph{Tail of $X_0X_h$.} Using similar computation as in Section \ref{sec:expar} we have
\begin{align*}
  \esp \left[ \left( \frac{X_0}{x} \1{\{X_0 \leq \epsilon x\}} \frac{X_h}{\scalfunccev_h(x)}
    \right)^\delta \right] = \esp\left[ \frac{V_0^{\delta(1+\phi^h)}Z_0}{x^{\delta(1+\phi^h)}}
    \1{\{V_0Z_0 \leq \epsilon x\}} \right] \esp \left[ \rme^{\delta\sum_{j=0}^{h-1} \phi^j
      \epsilon_h-j} \right]\esp[Z_h^{\delta}] \; .
\end{align*}
Moreover, if $F_Z$ is the distribution function of $Z_0$, then 
\begin{align*}
  &\frac{1}{\pr(V_0>x)}\esp\left[ \frac{V_0^{\delta(1+\phi^h)}Z_0}{x^{\delta(1+\phi^h)}} \1{\{V_0Z_0
      \leq \epsilon x\}} \right]=\frac{1}{\pr(V_0>x)} \int \esp\left[
    \frac{V_0^{\delta(1+\phi^h)}u}{x^{\delta(1+\phi^h)}} \1{\{V_0 \leq \epsilon x/u\}}  \right]F_Z(\rmd u)  \\
  &\leq \epsilon^{\delta(1+\phi^h)}\int u^{-\delta(1+\phi^h)}\frac{\pr(V_0>\epsilon x/u) }
  {\pr(V_0>x)} u F_Z(\rmd u)\leq C \epsilon^{\delta(1+\phi^h)-\alpha} \int
  u^{\alpha+1-\delta(1+\phi^h)} F_Z(\rmd u) \; .
\end{align*}
If $\delta$ is chosen such that $
\delta<\alpha$, $\alpha<\delta (1+\phi^h)<\alpha+1$, then the condition 
(\ref{eq:condition-epsilon-product-cev}) holds. 
\subsection{Stochastic volatility process with heavy tailed innovation}
Assume that $X_{t} = \sigma_{t}Z_t$, where $\sequence{Z}{t}{\Zset}$
is an i.i.d. sequence with regularly varying marginal distribution
with tail index $\alpha$, $\sigma_{t}$ is non
negative, $\esp[\sigma_{t}^q]<\infty$ for some $q>\alpha$ and
$\{\sigma_{t}\}$ and $\{Z_{t}\}$ are independent.  Then, by
Breiman's lemma, $X_{t}$ is regularly varying, and has extremal
independence:
\begin{align*}
  \pr(X_0>x) & \sim \esp[\sigma_0^{\alpha}]\bar F_Z(x) \; ,
\end{align*}
and
\begin{align*}
  \pr(X_0 > x \;, \ X_h>x)
  = o(\bar F_Z(x)) \; ,
\end{align*}
where $F_Z$ is the distribution function of $Z_0$.
 For any integer $h>0$, we have
\begin{align}
  \label{eq:sv-noL}
  \lim_{x\to\infty} \pr \left( \frac{X_0}x>y_0, X_1\leq y_1,\dots,X_h \leq y_h
    \mid X_0>x \right) = \frac{ \esp[\sigma_0^{\alpha} F_Z({y_1}/{\sigma_1})
    \cdots F_Z({y_h}/{\sigma_h})]} {y_0^{\alpha} \, \esp[\sigma_0^{\alpha}]} \; .
\end{align}
In particular,
\begin{align*}
  \lim_{x\to\infty} \pr(X_h \leq y_h \mid X_0>x) =
  \frac{\esp\left[\sigma_0^{\alpha} F_Z(y_h/\sigma_h) \right]}
  {\esp[\sigma_0^{\alpha}]} \; .
\end{align*}
The conditional scaling exponent $\kappa_h$ is 0 at all lags $h\ge 1$.  Note
also that
\begin{align*}
  \lim_{x\to\infty} \pr(X_h > y_h \mid X_0>x)=\frac{\esp\left[\sigma_0^{\alpha}
      \bar{F}_Z(y_h/\sigma_h) \right]} {\esp[\sigma_0^{\alpha}]} \sim \bar{F}_Z(y_h)
  \frac{\esp\left[\sigma_0^{\alpha}\sigma_h^{\alpha}\right]} {\esp[\sigma_0^{\alpha}]} \sim
  \bar{F}_X(y_h) \frac{\esp\left[\sigma_0^{\alpha}\sigma_h^{\alpha}\right]}
  {\esp[\sigma_h^{\alpha}]\esp[\sigma_0^{\alpha}]}\; ,
\end{align*}
as $y_h\to\infty$. Hence, the limiting conditional distribution is
tail equivalent to the unconditional distribution.

For more details on the extremal behavior of this model we refer to
\cite{davis:mikosch:2001} %, \cite{kulik:soulier:2011}, \cite{kulik:soulier:2012},
 and \cite{kulik:soulier:2013}. In particular, in the latter paper the
conditional model and its extensions to different conditioning events is
considered (cf. the discussion in Section \ref{sec:concluding}), together with
relevant statistical inference.

If $\alpha>1$, then Lemma~\ref{lem:weakconv-unbounded}  applies and we obtain
\begin{align*}
  \lim_{x\to\infty} \esp[(X_h)_+\mid X_0>x] = \frac{\esp[(Z_0)_+]
    \esp[\sigma_h \sigma_0^\alpha]} {\esp[\sigma_0^\alpha]} \; .
\end{align*}

\subsection{Stochastic volatility process with heavy tailed innovation and leverage}
\label{sec:sv-heavy-innovation-leverage}
We now consider a stochastic volatility process $X_{t} = \sigma_{t}Z_{t}$ and
assume that the volatility~$\sigma_{t}$ has the form
\begin{align*}
  \sigma_{t} = \sigma(\xi_{t}) \; ,
\end{align*}
where $\sigma$ is a positive function, $\xi_{t}=\sum_{j=1}^\infty c_j
\eta_{t-j}$, $\sum_{j=1}^\infty c_j^2<\infty$ and $\{(Z_{t},\eta_{t})\}$ is an
i.i.d. sequence, but for each $t$, $Z_{t}$ and $\eta_{t}$ may be dependent. This
implies that the volatility $\sigma_{t}$ is independent of the innovation
$Z_{t}$ for each $t$, but $\sigma_{t}$ may be dependent of $\{Z_j,j<t\}$. This
allows for some leverage: today's value impacts future volatility. We still
assume that the distribution of $Z_0$ is regularly varying with index
$\alpha$. For each~$t$, $Z_{t}$ and $\sigma_{t}$ are independent, thus, if
$\esp[\sigma_{t}^{q}]<\infty$ for some $q>\alpha$, Breiman's lemma applies and
we obtain
\begin{align*}
  \pr(X_0>x) & \sim \esp[\sigma_0^{\alpha}]\bar F_Z(x) \; .
\end{align*}
Consider now the probability of joint exceedances.  Since $\sigma_h$
and $Z_0$ may be dependent, we have,
\begin{align*}
  \pr(X_0 > x \;, \ X_h>x)
  & = \pr( Z_0 \sigma_0 > x \; , \ Z_h \sigma_h > x) \\
  & = \esp \left[
    \bar F_Z(x/\sigma_h) \1{\{Z_0 \sigma_0 > x\}} \right]      \\
  & \sim \bar F_Z(x) \esp \left[ \sigma_h^\alpha \1{\{Z_0 \sigma_0 > x\}}
  \right] = o(\bar F_Z(x)) \; .
\end{align*}
(For the last part, a bounded convergence argument is used.)  Thus there is
still extremal independence, as in the previous model with no leverage, but the
rate of decay of the joint exceedances probability is affected by the dependence
between $\sigma_h$ and $Z_0$.

Under additional assumptions, we can obtain the limiting conditional
distributions.  Assume that $\eta_{j} = \log(|Z_{j}|) - \esp[\log(|Z_{j}|)]$,
$\sigma(x)=\rme^x$ and $0<c_{j}<1$ for all $j\geq1$. Define $\tilde\sigma_j =
\exp\{\sum_{i=1}^{j-1} c_i \eta_{j-i} -c_j\esp[\log(|Z_{0}|)] +
\sum_{i=j+1}^\infty c_i \eta_{j-i}\}$. Then, $X_j = \tilde\sigma_j |Z_0|^{c_j}
Z_j$ and by the same type of arguments as previously, we obtain, for
$(y_0,\dots,y_h)\in[1,\infty)\times\Rset^h$,
\begin{multline*}
  \lim_{x\to\infty} \pr(X_0>xy_0,X_1\leq x^{c_1}y_1,\dots,X_h\leq x^{c_h} y_h  \mid X_0>x)  \\
  = \int_0^\infty \pr(\sigma_0>y_0u^{-1},Z_1\tilde\sigma_1\leq
  y_1u^{-c_1}, \dots, Z_h \tilde\sigma_h \leq y_h u^{-c_h} ) \, \alpha u^{-\alpha-1} \,
  \rmd u \; .
\end{multline*}
The conditional scaling exponent depends on $h$: $\kappa_h=c_h$. The
marginal limiting distributions are also tail equivalent to the distribution
of~$X_0$.

If $\alpha>1$, Lemma~\ref{lem:weakconv-unbounded} applies again and we obtain
\begin{align*}
  \lim_{x\to\infty} \esp \left[ \frac{(X_h)_+}{x^{\kappa_h}} \mid X_0 > x \right] =
  \frac{\alpha \esp[(Z_h)_+] \, \esp[\tilde\sigma_h \sigma_0^{\alpha-\kappa_h}]} {(\alpha
    - \kappa_h) \esp[\sigma_0^\alpha]} = \frac{\alpha \esp[(Z_h)_+] \, \esp[\sigma_h
    \sigma_0^{\alpha-\kappa_h}]} {(\alpha - \kappa_h) \esp[|Z_0|^{\kappa_h}] \esp[\sigma_0^\alpha]} \; .
\end{align*}

\section{Proof of Theorem~\ref{theo:markov-indep}}
\label{sec:prooftheomarkov}

The following result is related to \cite[Theorem 5.5, page 34]{billingsley:1968}
and is sometimes referred to as \textit{the second continuous mapping
  theorem}. See also \cite[Lemma~8.4]{resnick:zeber:2013:asy}.
\begin{theorem}
  \label{theo:our-2CMT}
  Let $(E,d)$ be a complete locally compact separable metric space.  Let $\mu_n$ be
  a sequence of probability measures which converge weakly to a probability
  measure $\mu$ on $E$.
  \begin{enumerate}[(i)]
  \item \label{item:unif-measure} If $\varphi_n$ is a uniformly bounded sequence of
    continuous functions which converge uniformly on compact sets of $E$ to a
    function $\varphi$, then $\varphi$ is continuous and bounded on $E$ and
    $\lim_{n\to\infty} \mu_n(\varphi_n) = \mu(\varphi)$.
  \item \label{item:unif-noyau} Let $F$ be a topological space. If $g_n$ is a
    sequence of uniformly bounded, continuous functions on $F \times E$ which
    converge uniformly on compact sets of $F \times E$ to a function $g$, then
    $g$ is continuous and bounded on $F \times E$ and the sequence of functions
    $\int_E g_n(u,v) \mu_n(\rmd v)$ converge uniformly on compact sets of $F$ to
    $\int_E g(u,v) \mu_n(\rmd v)$
  \end{enumerate}
\end{theorem}

\begin{proof}
  We start by proving~\eqref{item:unif-measure}. Let $C$ be such that $\sup_{n\geq1}
  \|\varphi_n\|_\infty \leq C$ and $\|\varphi\|_\infty\leq C$.  Fix some $\epsilon>0$ and let $K$ be
  a compact set such that $\mu(\partial(K^c))=0$ and $\mu(K^c)\leq \epsilon/(2C)$. Let $K_\epsilon=
  \{x\in E\mid d(x,K)\leq \epsilon\}$ and let $\psi$ be a continuous function such that $0\leq
  \psi(x)\leq 1$ for all $x\in E$, $\psi(x)=1$ if $x\in K_\epsilon$ and $\psi(x)=0$ if $x\notin
  K_\epsilon$.
  \begin{align*}
    \mu_n(\varphi_n) - \mu(\varphi) = \mu_n(\varphi_n) - \mu_n(\varphi) + \mu_n(\varphi) - \mu(\varphi) \; .
  \end{align*}
  By weak convergence, $\lim_{n\to\infty} \mu_n(\varphi) = \mu(\varphi)$, so we only need
  to consider $\mu_n(\varphi_n)-\mu_n(\varphi)$. Using the function $\psi$
  defined above, we have
  \begin{align*}
    | \mu_n(\varphi_n)-\mu_n(\varphi) | & \leq |\mu_n(\varphi_n \psi) -\mu_n(\varphi\psi) | + |\mu_n((1-\psi)\varphi_n) -
    \mu_n(\varphi(1-\psi))| \\
    & \leq \mu_n(|\varphi\psi-\varphi_n\psi|) + 2 C \mu_n(1-\psi) \; .
  \end{align*}
  Since $\varphi_n$ converges to $\varphi$ uniformly on compact sets and the function $1-\psi$
  is bounded and continuous, we obtain
  \begin{align*}
    \limsup_{n\to\infty} | \mu_n(\varphi_n)-\mu_n(\varphi) | & \leq 2 C \mu(1-\psi) \leq 2C
    \mu(K^c) \leq \epsilon \; .
  \end{align*}
Since $\epsilon$ is arbitrary, the proof of \eqref{item:unif-measure} is concluded.

We now prove~\eqref{item:unif-noyau}. Define $L_n(u) = \int_{E} g_n(u,v)
\mu_n(\rmd v)$, $\bar{L}_n(u)= \int_{E} g(u,v) \mu_n(\rmd v)$ and $L(u) =
\int_{E} g(u,v) \mu(\rmd v)$. Since $g$ is bounded and continuous, the first
part of the theorem implies that $\bar{L}_n$ converges uniformly to $L$ on
compact sets of $F$.  We now prove that $L_n-\bar L_n$ converges to zero
uniformly on compact sets of $F$. Fix $\epsilon>0$ and let $K_\epsilon$ be as
above. Since $g_n$ and $g$ are uniformly bounded, there exists $C>0$ such that
\begin{align*}
  |L_n(u) - \bar L_n(u)| & \leq  \sup_{v\in K_\epsilon} |g_n(u,v)-g(u,v)|
  + 2 C \epsilon \; .
\end{align*}
For any compact set $S$ of $F$, $g_n$ converges uniformly on $S\times
K_\epsilon$ to $g$, thus
\begin{align*}
  \limsup_{n\to\infty} \sup_{u\in S}   |L_n(u) - \bar L_n(u)| \leq 2 C \epsilon \; .
\end{align*}
Since $\epsilon$ is arbitrary, this proves that $L_n-\bar L_n$ converges to $0$
uniformly on compact sets of~$F$.
\end{proof}

We finally need the following lemma. Let $\Pi$ and $G$ be as in
Assumption~\ref{hypo:markov-kernel} and define the kernels $\Pi_x$ and $G_{1}$
by
\begin{align*}
  \Pi_xf(u)  & = \int_0^\infty f(v) \Pi(xu,b(x)\rmd v) \; , \\
  G_{1}f(u) & = \int_0^\infty f(u^\kappa v) G(\rmd v)=\int_0^\infty f(v) G(u^{-\kappa}\rmd v) \; .
\end{align*}
\begin{lemma}
  \label{lem:enfin}
  Let $f$, $f_x,x>0$, be uniformly bounded, continuous functions on
  $[0,\infty)$. Assume that
  \begin{enumerate}[(i)]
  \item \label{item:G0>0} either $f_x$ converges uniformly on compact sets of
    $[0,\infty)$ to $f$;
  \item \label{item:G0=0} or $f_x$ converges uniformly on compact sets of
    $(0,\infty)$ to $f$ and $G(\{0\})=0$.
  \end{enumerate}
  Then $\Pi_x f_x$ converges uniformly on compact sets of $(0,\infty)$ to
  $G_{1}f$.
\end{lemma}
\begin{proof}
  Fix some positive real numbers $0<a_0<a_1$. Since $b$ is regularly varying at
  infinity with positive index, without loss of generality, we can assume that
  $b$ is increasing and positive on $(a_0,\infty)$. Then, the ratio $b(xu)/b(x)$
  is uniformly bounded on $[a_0,a_1]$, i.e.
  \begin{align}
  \label{eq:lem-born-b}
    0 < \sup_{x\geq1} \sup_{a_0 \leq u \leq a_1} \frac{b(xu)}{b(x)} < \infty \; .
  \end{align}
  Fix some $\epsilon>0$. Then, there exists $A_\epsilon$ such that
  \begin{align}
  \label{eq:lem-born-infty}
    \limsup_{x\to\infty} & \sup_{a_0 \leq u \leq a_1}
    \Pi(xu,(b(x)A_\epsilon,\infty)) \leq \epsilon \; , \ \ \sup_{a_0 \leq u \leq
      a_1} G((u^\kappa A_\epsilon,\infty)) \leq \epsilon \; .
  \end{align}
  Moreover, if $G(\{0\})=0$, then there also exists $\eta>0$ such that
  \begin{align}
  \label{eq:lem-born-0}
    \limsup_{x\to\infty} & \sup_{a_0 \leq u \leq a_1} \Pi(xu,[0,b(x)\eta]) \leq
    \epsilon \; , \ \ \sup_{a_0 \leq u \leq a_1} G([0,u^\kappa\eta]) \leq
    \epsilon \; .
  \end{align}
  Let now $f_x$ and $f$ be as in the statement of the lemma. Then, by the
  uniform boundedness assumption and by~(\ref{eq:lem-born-infty}), there exists
  $C>0$ such that, for $u\in[a_0,a_1]$,
  \begin{align*}
    \left| \Pi_xf_x(u) - \Pi_x f(u)\right| & \leq \int_0^{A_\epsilon} |f_x(v) -
    f(v)| \Pi(xu,b(x)\rmd v) + C \epsilon \; .
  \end{align*}
  In case \eqref{item:G0>0}, it is assumed that $f_x$ converges uniformly on the
  compact sets of $[0,\infty)$, thus the previous bound yields
\begin{align*}
  \limsup_{x\to\infty} \sup_{a_0\leq u \leq a_1} \left| \Pi_xf_x(u) - \Pi_xf(u)\right|
    \leq C\epsilon \; .
\end{align*}
Since $\epsilon$ is arbitrary, this yields
\begin{align}
  \label{eq:lem-borne-fx}
  \limsup_{x\to\infty} \sup_{a_0\leq u \leq a_1} \left| \Pi_xf_x(u) - \Pi_xf(u)
  \right| = 0 \; .
\end{align}
In case~\eqref{item:G0=0}, we must further decompose the integral
into $\int_0^{A_\epsilon}=\int_0^{\eta}+\int_{\eta}^{A_\epsilon}$
and use the bound~(\ref{eq:lem-born-0}) to
obtain~(\ref{eq:lem-borne-fx}).

We now prove that $\Pi_xf$ converges uniformly on compact sets of
$(0,\infty)$ to $G_{1} f$. Define the function $f_t$ on
$(0,\infty)\times[0,\infty)$ by $f_t(u,v) = f(vb(t)/b(t/u))$.  By
the uniform convergence for regularly varying functions, $f_t$
converges uniformly on compact sets of $(0,\infty)\times[0,\infty)$
to $f(u^\kappa v)$.  Define $F_t(u) = \int_0^\infty f_t(u,v)
\Pi(t,b(t) \rmd v)$.  By item~\eqref{item:unif-noyau} of
Theorem~\ref{theo:our-2CMT}, $F_t$ converges to $G_{1} f$ uniformly
on compact sets of $(0,\infty)$. Note that by change of variables we
can write
\begin{align*}
  \Pi_xf(u) & = \int f(v)\Pi(xu,b(x)\rmd v) = \int f\left(\frac{vb(xu)}{b(x)}\right) \Pi(xu,b(xu)  \rmd v)  \\
  & = \int f_{xu}(u,v)\Pi(xu,b(xu)\rmd v)=F_{xu}(u)\; ,
\end{align*}
so $\Pi_xf$ converges uniformly on compact sets of $(0,\infty)$ to
$G_{1}$.

\end{proof}

\begin{proof}[Proof of Theorem~\ref{theo:markov-indep}]
  We must prove that for all $h\geq1$,
  \begin{multline}
    \label{eq:markov-tail-expar}
    \lim_{x\to\infty} \pr \left(\frac{X_0}x \leq y_0, \frac{X_1}{b_1(x)} \leq
      y_1, \dots, \frac{X_h}{b_h(x)} \leq y_h \mid X_0 > x \right) \\
    = \pr(Y_0 \leq y_0, Y_1 \leq y_1,\dots, Y_h \leq y_h) \; .
  \end{multline}
  The proof is by induction on $h$. We start by
  proving~(\ref{eq:markov-tail-expar}) in the case $h=1$. Recall that $F_0$ is
  the distribution of $X_0$ and define the measure $\nu_x$ by $\nu_x(\rmd u) =
  F_0(x\rmd u)/\bar F_0(x)$.  Let $f$ be a bounded continuous function on
  $[0,\infty)$. Then
\begin{align*}
  \esp\left[f\left(\frac{X_1}{b(x)}\right)\mid X_0>x\right] & =
  \int_{u=1}^\infty \int_{v=0}^\infty f(v) \Pi(xu,b(x)\rmd v) \nu_x(\rmd u) =
  \int_{u=1}^\infty \Pi_x f(u) \nu_x(\rmd u) \; ,
\end{align*}
and thus we must prove that
\begin{align}
  \label{eq:whatwewantotprove}
  \lim_{x\to\infty} \int_1^\infty \Pi_xf(u) \nu_x(\rmd u) = \int_1^\infty
  G_{1} f(u) \alpha u^{-\alpha-1} \, \rmd u \; .
\end{align}
We know that the measure $\nu_x$ converges vaguely on $(0,\infty)$
to the Pareto measure $\alpha u^{-\alpha-1} \, \rmd u$. By
Lemma~\ref{lem:enfin} (applied with $f_x=f$, thus not requiring the
assumption $G(\{0\})=0$), $\Pi_x f$ converges to $G_{1} f$ uniformly
on compact sets of $(0,\infty)$.  Thus, applying
Theorem~\ref{theo:our-2CMT}(i) we
obtain~(\ref{eq:whatwewantotprove}).

Consider now the higher dimensional distributions. Define the transition kernel
$\hpi{x}{h}$ on $[0,\infty)\times[0,\infty)^h$ by
\begin{align*}
  \hpi{x}{h} (u_0,\rmd\boldsymbol{u})= \prod_{i=1}^h \Pi(b_{i-1}(x)u_{i-1},b_i(x) \rmd u_i) \;
  ,
\end{align*}
with the convention $b_0(x)=x$.  For $f$  bounded and continuous on
$[0,\infty)^h$, define
\begin{align*}
  \hpi{x}{h}f(u_0) = \int_{\boldsymbol{u}\in[0,\infty)^h} f(\boldsymbol{u})\hpi{x}{h}(u_0,\rmd \boldsymbol{u}) \; .
\end{align*}
Then,
\begin{align*}
  \esp & \left[ f \left(\frac{X_1}{b_1(x)},\dots,\frac{X_h}{\scalfunccev_h(x)}\right)  \mid X_0 > x \right]
 = \int_{u_0=1}^\infty \hpi{x}{h} f (u_0)  \nu_x(\rmd u_0) \; .
\end{align*}
Define also the kernel $G_h$ on $(0,\infty)\times[0,\infty)^h$ by
\begin{align*}
  G_h f(u_0) & = \int_0^\infty \cdots \int_0^\infty f(u_1,\dots,u_h)
  \prod_{i=1}^h G(u_{i-1}^{-\kappa}\rmd u_i) \; .
\end{align*}
for any bounded continuous function $f$.
What we must prove is that $\hpi{x}{h}f$ converges uniformly on compact sets of
$(0,\infty)$ to $G_hf$ for any bounded continuous function $f$ on~$[0
,\infty)^h$. By Theorem~\ref{theo:our-2CMT}(i), this will yield the required
result. For $h=1$ this is what we have just proved.  Assume now that for
$h\geq1$, and any bounded continuous function $f$ on $[0,\infty)^h$,
$\hpi{x}{h}f$ converges uniformly on compact sets of $(0,\infty)$ to
$G_hf$. Without loss of generality, we can assume that the function $f$ is of
the form $f(u_1,\dots,u_{h+1})= f_1(u_1)f_2(u_2,\dots,u_{h+1})$, where $f_1$ and
$f_2$ are continuous and bounded on $[0,\infty)$ and $[0,\infty)^h$,
respectively. Then, recalling that $b_h = b_{h-1} \circ b$,
\begin{align}
  \label{eq:induction-h}
  \hpi{x}{h+1}f(u_0) & = \int_0^\infty f_1(u_1) \hpi{b(x)}{h}f_2(u_1)
  \Pi(xu_0,b(x)\rmd u_1)  = \Pi_x (f_1\hpi{x}{h}f_2)(u_0) \; .
\end{align}
By the induction assumption, the sequence functions $f_1\hpi{x}{h}f_2$ converges
uniformly on compact sets of $(0,\infty)$ to the continuous and bounded function
$f_1 G_hf_2$.  Thus, by Lemma~\ref{lem:enfin}, $\hpi{x}{h+1}f$ converges to
$G_1(f_1G_hf_2)=G_{h+1} f_1f_2=G_{h+1}f$ uniformly on the compact sets of
$(0,\infty)$.
\end{proof}

\section{Concluding remarks}
\label{sec:concluding}
In this paper, we have put the concept of conditional extreme values in the
context of univariate time series. We have introduced the conditional scaling
exponent $\kappa_h$ of a time series $\{X_t\}$ at lag $h$. If the time series is
stationary and its finite dimensional marginal distributions are regularly
varying, then $\kappa_h\in[0,1]$ and a value $\kappa_h<1$ implies the extremal
independence of the bivariate distribution $(X_0,X_h)$.  We have given
conditions for Markov chains  and other time series models commonly used
in financial econometrics. This work is part of an ongoing project on extremally
independent time series. We now briefly discuss several possible future lines of
research.

\paragraph{Vector valued time series.}
Consider a $d$-dimensional vector valued times series
$\sequence{\boldsymbol{X}}{t}{\Zset}$ such that for each $h\geq0$, the
$(h+1)d$-dimensional vector $(\boldsymbol{X}_0,\dots,\boldsymbol{X}_h)$ is
regularly varying with index $-\alpha$. For a relatively compact Borel set
$C\in\overline{\Rset}^{h+1}\setminus\{\boldsymbol0\}$ (possibly with further regularity conditions), we may be
interested in the limiting distribution of
$(\boldsymbol{X}_1,\dots,\boldsymbol{X}_h)$ given that $\boldsymbol{X}_0 \in x C$,
where $xC = \{x\boldsymbol{y}, \boldsymbol{y}\in C\}$ and $x$ is large. In the
case of extremal dependence, the exponent measure of the vector
$(\boldsymbol{X}_0,\dots,\boldsymbol{X}_h)$ provides the necessary
information. In the case of extremal independence, it is useless, and we must
investigate the existence of scaling functions $b_1,\dots,\scalfunccev_h$ such
that the conditional distribution of
\begin{align*}
  \left( \frac{\boldsymbol{X}_1}{\scalfunccev_1(x)}, \dots, \frac{\boldsymbol{X}_h}{\scalfunccev_h(x)} \right)
\end{align*}
given $\boldsymbol{X}_0 \in xC$ converges to a proper probability distribution.
The choice of the set $C$ is determined by the problem considered. It could be
the complement of the unit ball for some norm $\|\cdot\|$ on $\Rset^d$, if the
event of interest is that $\|\boldsymbol{X}_0\|$ is large, or a half-space such
as $C=\{\boldsymbol{y}\in \Rset^d \mid a_1y_1+\cdots+a_ky_d>1\}$, if the event
of interest is that a certain linear combination (a portfolio) is large.

\paragraph{Different conditioning events.} We can also consider univariate time
series and various conditioning events such as $\{y_0>x,\dots, y_k>x\}$ ($k+1$
successive values are large), or $\{\max\{y_0,\dots,y_k\}>x\}$ (at least one
large value among the first $k+1$), or any combination of such events. Again, in
the case of extremal dependence the appropriate scaling is given by the
multivariate regular variation property and the entire information is given by
the exponent measure.  In the case of extremal independence different scaling
functions must be used for different lags and the limiting distributions are not
given by the exponent measure.

\paragraph{Statistical procedures.}
The next step is obviously to provide valid statistical procedures to estimate
the conditional scaling exponents, the scaling functions, the conditional
limiting distributions and other quantities of interests such as the CTE.  As
usual in extreme value theory, these quantities cannot be estimated empirically
since they are relevant only in a domain where few observations are
available. Therefore extrapolation outside the range of available data is needed
and semiparametric estimators must be defined.  For instance, one can estimate
$m_h$ (defined in~(\ref{eq:def-mh})) and $\kappa_h$ and then to estimate
$\CTE_h^+(x)$ for $x$ outside the range of the data by
 \begin{align*}
   \widehat{\CTE}^{\rm{SP}}_h(x) =  x^{\hat \kappa_h} \hat{m}_h \; .
 \end{align*}
 Non parametric estimators of the conditional limiting distributions and of the
 scaling functions, as well as semiparametric estimators of the conditional
 scaling exponents are the subject of our future research.

\section*{Acknowledgment}
The authors are grateful to Holger Drees and Anja Janssen for the communication
of the Reference~\cite{janssen:drees:2013} and for fruitful conversations during
the 8th EVA conference in Shanghai which allowed to correct and improve a
preliminary version of this paper. We are also grateful to an associate editor
and an anonymous referee who helped reorganize and improve the paper.

\end{document}